\documentclass[10pt,reqno]{amsart}
\usepackage{amsmath}
\usepackage{mathtools}
\numberwithin{equation}{section}
\usepackage[ngerman, english]{babel}
\usepackage[T1]{fontenc}
\usepackage{microtype}
\usepackage{amssymb}
\usepackage{enumitem}
\usepackage{tikz}
\usepackage{cancel}
\usepackage{graphicx}
\usepackage{esint}
\usepackage[utf8]{inputenc}
\usepackage{hyperref}
\usepackage{relsize}
\usepackage[T1]{fontenc}
\usepackage{tikz}
\usepackage{lmodern}
\usepackage{bigints}
\usepackage{relsize}
\newtheoremstyle{boldremark}     
{\topsep}                      
{\topsep}                      
{\normalfont}                  
{0pt}                          
{\bfseries}                    
{.}                            
{ }                            
{}                             
\DeclareUnicodeCharacter{0301}{\'{u}}
\usetikzlibrary{patterns}
\allowdisplaybreaks
\usepackage{bbm}
\newtheoremstyle{break}
{}
{}
{\itshape}
{}
{\bfseries}
{.}
{\newline}
{}

\newtheorem{theorem}{Theorem}[section]
\newtheorem{corollary}{Corollary}[section]
\newtheorem{definition}{Definition}[section]
\newtheorem{proposition}{Proposition}[section]
\newtheorem{lemma}{Lemma}[section]

\newtheorem{example}{Example}[section]
\newtheorem{remark}{Remark}[section]

\newcommand{\dint}[1]{\mathrm{d} #1}

\newcommand{\limesinf}[2]{\displaystyle \liminf_{#1 \rightarrow #2}}

\newcommand{\limes}[2]{\displaystyle \lim_{#1 \rightarrow #2}}

\newcommand{\norm}[2]{\left\|#1 \right\|_{#2}}

\newcommand{\indicator}[1]{\mathds{1}_{#1}}

\newcommand{\scalprod}[2]{\langle #1,#2 \rangle}

\newcommand{\palme}{\Upsilon}
\newcommand{\call}[1]{\mathcal{#1}}


\newcommand{\bbN}{\mathbb{N}}

\newcommand{\bbR}{\mathbb{R}}
\newcommand{\bbS}{\mathbb{S}}

\newcommand{\calA}{\mathcal{A}}
\newcommand{\calB}{\mathcal{B}}

\newcommand{\calF}{\mathcal{F}}

\usepackage[T1]{fontenc}
\usepackage{lmodern}
\usepackage{dsfont}

\newcommand{\mengdiff}[2]{#1\setminus #2}
\newcommand{\proj}[2]{\Pi_{#1}^{#2}}
\usepackage{comment}
\usepackage[backend  = biber,giveninits = true,style = numeric]{biblatex}
\renewbibmacro{in:}{}

\DeclareFieldFormat[article]{title}{#1}
\AtEveryBibitem{%
  \clearfield{doi}%
  \clearfield{eprint}%
  \clearfield{number}%
  \clearfield{issn}
  \clearfield{urldate}
}

\DeclareFieldFormat{pages}{#1}
\DeclareFieldFormat{pagetotal}{#1}
\AtEveryBibitem{\footnotesize}

\begin{document}

\title{Nonlocal Fisher information: lifting, local limit, and the Blachman-Stam inequality}

\author{Fabian Merz$^{a}$}
\email{fabian.merz@uni-ulm.de}
\author{Rico Zacher$^{*,a}$}
\thanks{$^*$Corresponding author}
\email[Corresponding author:]{rico.zacher@uni-ulm.de}
\address{$^a$ Institut f\"ur Angewandte Analysis, Universit\"at Ulm, Helmholtzstra\ss{}e 18, 89081 Ulm, Germany.}


\begin{abstract}
We show that the nonlocal Fisher 
information -- defined as the entropy dissipation of the Boltzmann entropy for nonlocal heat equations -- admits a natural lifting in the sense of Guillen and Silvestre \cite{guillen_landau_2025-1}. Important examples include the discrete Fisher information arising in Markov chains and the fractional Fisher information $i_s$ associated with the fractional Laplacian
$(-\Delta)^{s}$ on $\bbR^d$, $s\in (0,1)$. 
We further establish a Blachman-Stam inequality (BSI) for the fractional Fisher information $i_s$, and prove that, for a large class of functions, $i_s$ converges to the classical Fisher information as $s\to 1$. Through this nonlocal-to-local limit, we recover the classical BSI and the lifting property of the classical Fisher information.    
\end{abstract}

\maketitle


\vspace{12pt}

\section{Introduction}
\subsection{Lifting the classical Fisher information and kinetic theory} \label{SubSec1.1}
\noindent In two seminal works, it was recently shown that the Fisher information 
\begin{equation} \label{FishCla}
i(f): = \displaystyle \int_{\bbR^{d}} f(x) |\nabla (\log f(x))|^{2}\ \dint{x}
\end{equation}
is nonincreasing along solutions to the spatially homogeneous Landau ($d=3$) and Boltzmann equation, 
 see \cite{guillen_landau_2025} and \cite{imbert_monotonicity_2026}. This striking monotonicity holds for a wide class of interaction potentials (in the Landau case) and collision kernels (in the Boltzmann case), covering all physically relevant power-law interactions.
An important consequence is the global existence of smooth solutions.

In order to motivate the main goal of this paper, let us briefly describe some of the key ideas developed in \cite{guillen_landau_2025} and \cite{imbert_monotonicity_2026} in order to prove 
the monotonicity of the Fisher information. Recall that both equations are of the form
\[
\partial_t f=Q(f),
\]
with a nonlinear and nonlocal collision operator $Q$. We write $Q_L$ for the Landau and $Q_B$ for the Boltzmann operator. 

The first crucial observation is that in both cases, $Q(f)$ can be expressed as the evaluation of a projection of a linear operator, acting on functions defined on $\bbR^{2d}$, evaluated at $f \otimes f$, where for a function $f:\bbR^{d} \rightarrow \bbR$, $ f \otimes f$ is given by $(f \otimes f)(v,w) = f(v) f(w),\ v,w \in \bbR^{d}$. For instance, in case of the Boltzmann operator, it is not difficult to see that
\[ 
Q_{B}(f)(v) = \int_{\bbR^{d}} \mathcal{B}(f \otimes f)(v,v_{\ast})\ \dint{v_{\ast}},
\]
where
\[
\mathcal{B}(F)(v,v_{\ast}) = \int_{\bbS^{d-1}} \big(F(v^{\prime},v_{\ast}^{\prime}) - F(v,v_{\ast})\big)\cdot B(v- v_{\ast},\sigma)\ \dint{\sigma}.
\]
Here, the post-collisional velocities $v',v_\ast'$ are given in terms of 
the pre-collisional velocities $v,v_\ast$ 
and the scattering direction $\sigma$.

The second key observation is the following. Consider the Fisher information $I$, defined on a suitable class of nonnegative functions on $\bbR^{2d}$, i.e.
\[
I(F)=\displaystyle \int_{\bbR^{d}} \int_{\bbR^{d}}F(x,y)|\nabla (\log F(x,y))|^{2}\ \dint{x}\,\dint{y}.
\]
Suppose that $I$ is nonincreasing along solutions
$F(s,v,w)$ of the initial value problem \begin{align}\label{Fproblem}\begin{cases}
	\partial_{s}F = \mathcal{B}(F),\\
	F(0) = f(t,v) f(t,w),	
\end{cases}
\end{align}
where $f(t,\cdot)$ denotes an arbitrary solution to the space-homogeneous Boltzmann equation at a fixed time $t>0$.
Under this assumption, the Fisher information $i$ on $\bbR^{d}$ acts as a Lyapunov functional for the space-homogeneous Boltzmann equation. An analogous statement holds for the space-homogeneous Landau equation, with $\mathcal{B}$
replaced by the corresponding linear operator appearing in the representation of $Q_L$ (see \cite{guillen_landau_2025}).

This conclusion relies on two structural properties of the Fisher information: \begin{align*}
&\text{(A)}\quad I(f \otimes f) = 2\, i(f),\; \text{ for suitable probability densities (pdfs) $f$ on $\bbR^{d}$} 
\intertext{and} 
&\text{(B)}\quad  I(F)  \geq 2 \, i(\proj{\lambda}{1}F),\; \text{ for suitable and symmetric pdfs $F$ on $\bbR^{2d}$.}\notag 
\end{align*}
Here, $\lambda$ denotes the $d-$dimensional Lebesgue measure and $\proj{\lambda}{1}F$ is the projection of $F$ on the first $d$ coordinates, i.e. $\proj{\lambda}{1}F(x) = \int_{\bbR^{d}} F(x,y)\ \dint{y}.$

Let $F$ be the solution of \eqref{Fproblem} and define $\theta:\,[0,\infty) \rightarrow \bbR$ via 
\[
\theta(s) = \frac{1}{2}I\big(F(s,\cdot)\big) - i\big(\proj{\lambda}{1}F(s,\cdot)\big).
\]
Then, by (A) and (B), $\theta$ attains its minimum at $s = 0$. Consequently, 
\begin{align*}
 0 & \leq  \frac{1}{2}\langle I^{\prime}(F), \partial_{s}F \rangle|_{s = 0} - \langle  i^{\prime}(\proj{\lambda}{1} F), \partial_{s} \proj{\lambda}{1} F\rangle|_{s = 0} \notag \\
& = \frac{1}{2}\langle I^{\prime}(f \otimes f), \mathcal{B}(f \otimes f) \rangle - \langle  i^{\prime}(\proj{\lambda}{1}(f \otimes f)), \proj{\lambda}{1}( \calB(f \otimes f))\rangle  \notag \\
& =  \frac{1}{2} \langle I^{\prime}(f \otimes f), \calB(f \otimes f)  \rangle - \langle i^{\prime}(f),Q_B(f) \rangle,
\end{align*}
which yields
\begin{align} \label{KeyInequ1}
& \langle i^{\prime}(f),Q_{B}(f) \rangle \leq \frac{1}{2}\langle I^{\prime}(f \otimes f),\calB(f \otimes f)  \rangle.
\end{align}
This inequality implies that $\partial_t i(f)\le 0$ whenever $\partial_s I(F)|_{s=0}\le 0$. For the latter, it suffices to show that
\begin{equation} \label{KeyInequ12}
\langle I'(F),\calB(F)\rangle\le 0 
\end{equation}
for every sufficiently smooth positive function $F$ on $\bbR^{2d}$ satisfying $F(x,y)=F(y,x)$. Proving the monotonicity of $i(f)$ via \eqref{KeyInequ12} is the strategy used in \cite{imbert_monotonicity_2026}, and analogously in \cite{guillen_landau_2025} in the Landau case. We remark that, in fact, equality holds in \eqref{KeyInequ1}, see \cite[Sec.\ 3]{imbert_monotonicity_2026}. 

The above argument with $\theta$ applies equally if the pair $(i,I)$ is replaced by any pair of functionals $(j,J)$ satisfying (A) and (B). In the recent work \cite{guillen_landau_2025-1}, Guillen and Silvestre introduced in this context the notion of a {\em lifting}
for functionals defined over probability densities on $\bbR^d$. Roughly speaking, a functional $J$ is called a lifting of the functional $j$ if properties (A) and (B) are satisfied. Guillen and Silvestre note that, in addition to the (classical) Fisher information, entropy also admits a natural lifting, and that it would be interesting to identify further functionals with this property.
\subsection{Nonlocal Fisher information} \label{SubSec1.2}
The main objective of this paper is to show that a large class of functionals, which can be understood as  {\em nonlocal Fisher information}, admits a natural lifting in the sense described above. Our setting is considerably more general than in \cite{guillen_landau_2025-1}, as we consider probability densities on arbitrary metric spaces. Important examples include the discrete Fisher information arising in Markov chains, as well as the fractional Fisher information $i_s$ associated with the fractional Laplacian
$(-\Delta)^{s}$ on $\bbR^d$, $s\in (0,1)$. Passing to the limit $s\to 1$, the lifting property of $i_s$ yields that of the classical Fisher information.

To describe our setting, let $(M,d)$ be a metric space, $\calB(M)$ its Borel $\sigma$-algebra, and 
$k:\,M \times \calB(M) \rightarrow \overline{\bbR}$ a {\em kernel} on $(M,d)$. This means that, for every $A \in \calB(M)$, the map $x \mapsto k(x,A)$ is Borel measurable, and for each fixed $x \in M$, the map $k(x,\cdot): \calB(M) \rightarrow [0,\infty]$ defines a $\sigma$-finite measure. 
With the kernel $k$, we associate the nonlocal operator $L$ defined by
\begin{align}
Lf(x) = \int_{M\setminus \{x\}} \big(f(y) - f(x)\big) \ k(x,\dint{y}),\quad x \in M, \label{defL} 
\end{align} 
for measurable functions $f:\, M \rightarrow \bbR$ such that the integral above exists for all $x \in M$. Following \cite{weber_liyau_2023}, we also allow
for the possibility that the integral in \eqref{defL} is singular. In this case, $Lf(x)$ is to be understood in the sense that $\int_{M\setminus \{x\}}$ is replaced by 
$\lim_{\varepsilon \to 0+} \int_{\mengdiff{M}{B_{\varepsilon}(x)}}$.

A major role in our framework is played by the operator $\Psi_{\palme}$ which is defined, for measurable $f:\,M\to   \bbR$, by
\begin{align}\Psi_{\palme}(f)(x) := \int_{\mengdiff{M}{\{x\}}}\palme(f(y) - f(x))\ k(x,\dint{y}),\quad x \in M, \label{Psipalme}\end{align}
where $\palme$ is given by $\palme(r) = e^{r}-1 -r,\, r \in \bbR$. We also write $\Psi_{\palme}^{k}$ to indicate the dependence on the kernel $k$. Since $\palme$ is nonnegative, the integral in \eqref{Psipalme} takes values in $[0,\infty]$.

The operator $\Psi_{\palme}$ was introduced in \cite{dier_discrete_2021} in the discrete setting in the context of Li-Yau inequalities on graphs. It plays a central role in \cite{weber_entropy_2021}, where a nonlocal Bakry-\'{E}mery theory was developed for discrete Markov chains. In the more general setting considered here, $\Psi_{\palme}$ also appears in \cite{weber_liyau_2023} and is crucial for establishing a reduction principle that enables the derivation of Li-Yau inequalities for general nonlocal operators.

The significance of the $\Psi_{\palme}$-operator lies in the fact that it provides, in several respects, a natural nonlocal analogue of the carr\'e du champ operator 
$\Gamma$ associated with a Markov generator, see \cite{bakry_analysis_2014} and Section \ref{Sec3} below for the definition of $\Gamma$. For example, the identity
\begin{equation}\label{eq:fundamentalidentity}
L (\log f) =\frac{L f }{f} - 
\Psi_\Upsilon(\log f),
\end{equation}
see \cite[Sec.\ 2]{weber_liyau_2023}, is the nonlocal counterpart of $\Delta (\log f)=\frac{\Delta f}{f}-|\nabla (\log f)|^2$. 

The quadratic term in the last relation also appears in the definition of the Fisher information $i(f)$ given in \eqref{FishCla}. Replacing this gradient term by $\Psi^k_\Upsilon(\log f)$ and $\bbR^d$ by $M$ formally leads to the {\em nonlocal Fisher information} defined by
\begin{equation} \label{NonlocFish}
i_{k}(f): = \int_{M} f(x) \Psi_{\palme}^{k}(\log f)(x)\ \dint{\mu(x)},
\end{equation}
for strictly positive measurable functions $f:\,M \rightarrow \bbR$. Here $\mu$ is a fixed $\sigma$-finite measure on $\calB(M)$ and it has to be ensured that 
$\Psi_{\palme}^{k}(\log f)$ is measurable. This is, for instance, guaranteed if $k$ is the countable sum of finite kernels, see \cite[Lemma 3.2]{kallenberg_foundations_2021}.
Observe that $i_{k}(f)\in [0,\infty]$.

A typical situation is that $\mu$ is an invariant and reversible measure for the Markov semigroup $(P_t)_{t\ge 0}$ generated by $L$ in a suitable functional-analytic setting under appropriate conditions on the kernel $k$. In this case, $i_k$
appears as the negative time derivative of entropy along the heat flow associated with the operator $L$, i.e.
\begin{equation} \label{entfish}
\partial_t h(P_t f)=-i_k(P_t f),\quad t\ge 0,
\end{equation}
for every suitable strictly positive probability density $f$ with respect to $\mu$, where the entropy $h(f)$ is given by
\begin{equation}
    h(f)=\int_M f \log f\,\dint{\mu}.
\end{equation}
We refer to \cite[Prop.\ 3.3]{weber_entropy_2021}
for a rigorous treatment in the discrete Markov chain setting, see also \cite[Prop.\ 5.10]{erbar_gradient_2014}
in the framework of jump processes on $\bbR^d$. Equation \eqref{entfish} is a nonlocal variant of
the classical entropy–entropy dissipation relation
$\partial_t h(P_t f)=-i(P_t f)$ for the heat flow generated by the Laplacian on $\bbR^d$.

The notion of nonlocal or {\em fractional} Fisher information has already been considered in \cite{erbar_gradient_2014,granero-belinchon_fractional_2017,rougerie_two_2020,salem_propagation_2019}, \cite{toscani_fractional_2016}. The definition used in \cite{granero-belinchon_fractional_2017} and \cite{salem_propagation_2019} corresponds to the special case of \eqref{NonlocFish} with $L=-(-\Delta)^s$, $s\in (0,1)$.
More general jump processes on $\bbR^d$ are studied in
\cite{erbar_gradient_2014}, where the concept of nonlocal Fisher information is introduced in a form that essentially coincides with our definition. We further refer to \cite{cover_elements_2006, erbar_ricci_2012,weber_entropy_2021} for the discrete Markov chain setting. We point out that in the literature the nonlocal Fisher information is often expressed in terms of products of the form $(b-a)(\log b-\log a)$, especially in symmetric or discrete settings. In fact, if $k(x,\dint{y}) \dint{\mu}(x)$ is symmetric, then $i_k(f)$ can be rewritten as
\begin{equation*}
i_k(f)=\frac{1}{2} \int_{M} \int_{ M} \big(f(y) - f(x)\big) \big(\log(f(y)) - \log(f(x))\big)  k(x,\dint{y})\, \dint{\mu(x)},
\end{equation*}
see \cite[p.\ 26]{weber_entropy_2021} for a corresponding computation in the setting of discrete Markov chains.
In view of the important role of the $\Psi_\Upsilon$-operator as a natural nonlocal analogue of 
$\Gamma$, we prefer the representation \eqref{NonlocFish}.
\subsection{Main results}
In this paper we establish several important properties of nonlocal Fisher information. 

Our first main result shows that, in the general setting described above, the functional $i_k$ admits a natural lifting, which is given by the nonlocal Fisher information $i_{k \oplus k}$, see Theorem \ref{ThmLiftingNonlocalFisher}. This functional is defined for suitable probability densities on the metric space $M\times M$ with respect to $\mu \otimes \mu$.
The kernel $k \oplus k$ denotes the tensorization of $k$ with itself and is defined in a natural way, see Definition \ref{TensorKernel}.

A key ingredient in the proof of Theorem \ref{ThmLiftingNonlocalFisher} is a convexity-type inequality for the nonlocal Fisher information operator $f\mapsto f \Psi_\Upsilon(\log f)$, due to Weber and Zacher \cite[Lemma 2.2]{weber_liyau_2023}, see Lemma \ref{main inequlity frederic} below. This fundamental inequality can be viewed as a nonlocal analogue (and generalization) of the inequality
\begin{equation} \label{gradientIE}
\int_{\bbR^d}\big|\nabla_x (\log H(x,y))\big|^2 H(x,y) f(y)\,\dint{y}\ge \big|\nabla (\log Pf(x))\big|^2 Pf(x),\;\; x\in \bbR^d,
\end{equation}
where $Pf(x)=\int_{\bbR^d}H(x,y)f(y)\,\dint{y}$ and $H$ and $f$ are sufficiently regular positive functions. 
The direct nonlocal counterpart of \eqref{gradientIE} reads 
\begin{equation} \label{PsiUpsIE}
\int_{\bbR^d} \Psi_\Upsilon (\log H(\cdot,y))(x)H(x,y)f(y)\,\dint{y} \geq \Psi_\Upsilon(\log P f)(x) P f(x),\;\; x\in \bbR^d,
\end{equation}
and appears as a particular case of Lemma \ref{main inequlity frederic}. Note that, in \cite{weber_liyau_2023}, the proof of a Li-Yau inequality $(-\Delta)^s (\log u)\le \frac{C(d,s)}{t}$ for positive solutions of the fractional heat equation $\partial_t u + (- \Delta)^s u = 0$
in $\bbR^d$ with $s\in (0,1)$ relies crucially on the inequality \eqref{PsiUpsIE}.

Our second main result, Theorem \ref{ThmBSI}, whose proof also makes essential use of \eqref{PsiUpsIE}, is a Blachman-Stam inequality (BSI) for the fractional Fisher information $i_s$ associated with the operator $L=-(- \Delta)^s$ on $\bbR^d$, $s\in (0,1)$. The corresponding kernel is given by
\begin{equation} \label{fractionalkernel}
    k_{s}(x,\dint{y}) =\frac{c(d,s)}{|x-y|^{d+2s}}\ \dint{y}\quad\text{with}\quad c(d,s)= \displaystyle \left(\int_{\bbR^{d}} \frac{1-\cos(\xi_{1})}{|\xi|^{d+2s}} \ \dint{\xi}\right)^{-1},
\end{equation}
see e.g.\ \cite[Sec.\ 3]{di_nezza_hitchhikers_2012}.
Our nonlocal BSI asserts that for every $\alpha\in (0,1)$ and probability densities $f,g$ on $\bbR^d$, one has 
\begin{align}
    i_{s}(f_{\sqrt{\alpha}} \ast g_{\sqrt{1-\alpha}}) \leq \alpha^{s}\,
     i_{s}(f) + (1-\alpha)^{s}\,   i_{s}(g), \label{BSI}
\end{align}
where $\ast$ denotes convolution in $\bbR^d$ and $f_{c}(\cdot) := {c^{-d}} f\left(\frac{\cdot}{c}\right)$ for a constant $c>0$.
Formally taking $s=1$ in \eqref{BSI} recovers the classical Blachman-Stam inequality for the Fisher information $i$ given by \eqref{FishCla}, see e.g.\ \cite[Sec.\ 1]{villani_fisher_2025}. This inequality is a 
fundamental tool in information theory, probability and the analysis of diffusion equations, as it provides a sharp control of the Fisher information under convolution. It plays an important role in the study of entropy, heat flows, and information-theoretic proofs of the central limit theorem (CLT); see, for instance, the seminal work of Barron \cite{barron_entropy_1986}. We expect that the nonlocal BSI \eqref{BSI} will be similarly useful in the context of jump processes and may serve as a key ingredient in an entropic proof of a corresponding CLT in the spirit of Barron.

The third main result, Theorem \ref{NonlocalToLocal}, shows that for a large class of positive functions $f$ on $\bbR^d$, one has 
\begin{equation} \label{NLTL}
i_s(f)\to i(f) \quad \text{as}\;\, s\to 1.
\end{equation}
Having this result at hand, we can rigorously take the limit $s\to 1$ in the nonlocal BSI, as well as in the conditions appearing in the definition of the lifting of $i_s$, thereby recovering the classical results. In particular, this demonstrates the robustness of our estimates for $i_s$ as $s\to 1$.

To the best of the authors' knowledge, all three results described above seem to be new. Upon completing this work, we became aware that the lifting property of the fractional Fisher information $i_s$ was already established in a stochastic context in \cite[Proposition 3.1 (iii)]{salem_propagation_2019}. The proof given there differs from ours -- it does not use \eqref{PsiUpsIE} -- and is restricted to the special case of the fractional Laplacian.

The paper is organized as follows. In Section \ref{Sec2}, we extend the definition of lifting from \cite{guillen_landau_2025} to metric spaces and establish the lifting property of the nonlocal Fisher information 
$i_k$. Section \ref{Sec3} discusses Fisher information in the framework of Bakry–Émery $\Gamma$-calculus and shows that certain weighted versions of the classical Fisher information also admit a lifting. Section \ref{Sec4} is devoted to the nonlocal-to-local limit \eqref{NLTL}, while Section \ref{Sec5} contains the proof of the nonlocal BSI. Finally, the appendix provides a proof of the lifting property of entropy along with several auxiliary results.
\section{Lifting property of nonlocal Fisher information}\label{Sec2}
\noindent We begin by introducing some basic concepts and notations. 

Given a metric space $(M,d)$, we denote its Borel $\sigma$-algebra by $\calB(M)$. Measurability of 
a function $f:\,M \rightarrow \overline{\bbR}$ will always refer to $\calB(M)$. For two metric spaces
$(M_{i},d_{i})$, $i = 1,2$, we equip $M_{1} \times M_{2}$ with the metric $d$ given by $d((x,y),(u,v)) = d_{1}(x,u) + d_{2}(y,v)$, and we always assume that $\calB(M_{1} \times M_{2}) = \calB(M_{1}) \otimes \calB(M_{2})$. This assumption is satisfied in the most interesting cases, e.g.\ if the topologies generated by the metrics $d_{i}$, $i=1,2$, have a countable basis, see e.g. \cite[Theorem 5.10]{elstrodt_mas-_2018}.

Let $(M,d)$ be a metric space and $\mu$ a measure on $\calB(M)$. For a function $F:\, M \times M \rightarrow \bbR$ we define the two projections with respect to $\mu$
\begin{align*}
\proj{\mu}{1}F(x) &= \int_{M} F(x,y) \ \dint{\mu(y)},\quad x \in M,\\
\proj{\mu}{2}F(y) &= \int_{M} F(x,y) \ \dint{\mu(x)},\quad y\in M.
\end{align*}
Note that if $F$ is symmetric, i.e.\ $F(x,y)=F(y,x)$ for all $x,y\in M$, then $\proj{\mu}{1}F(x)=\proj{\mu}{2}F(x)$.
In this case we just write $\Pi_\mu$ for both projections.
For two functions $f,g:\,M \rightarrow \bbR$ the tensorization of $f$ and $g$ is given by $$f \otimes g:M \times M \rightarrow \bbR,\ (x,y) \mapsto f(x) g(y).$$ 

We next extend the definition of a lifting given in \cite[Definition 1]{guillen_landau_2025-1} to our more general setting.
\begin{definition} \label{LiftingGenDef}
Let $(M,d)$ be a metric space and $\mu$ a $\sigma$-finite measure on $\calB(M)$. Let further $j(f)$ be a functional defined on a subset $D(j)$ of the probability densities on $M$ with respect to $\mu$, and let $J(F)$ be a functional defined on a subset $D(J)$ of the probability densities on $M\times M$ with respect to $\mu\otimes \mu$. We say that $J$ is a {\bf lifting} of $j$ if the following conditions hold.
\begin{itemize}
\item[(i)] For every $f\in D(j)$ with $f \otimes f\in D(J)$,
\[
J(f \otimes f) = 2j(f).
\]
\item[(ii)] For every symmetric $F\in D(J)$ with $\Pi_\mu F\in D(j)$,
\[
J(F) \geq 2j(\Pi_\mu F).
\]
\end{itemize}
\end{definition}
A few functionals are known to admit a natural lifting. As shown in \cite{guillen_landau_2025-1,guillen_landau_2025}, the classical Fisher information $i(f)$ given by \eqref{FishCla} enjoys this property. The authors of \cite{guillen_landau_2025} also mention (without proof) that entropy possesses a natural lifting. For the reader's convenience we provide an argument for this statement in the appendix, see Section \ref{A1}.

Note that a lifting of a functional $j$ need not be unique. Indeed, if $J$ is a lifting of the functional $j$ we can add any nonnegative functional $M$ satisfying $M(f \otimes f) = 0$. Then $J + M$ is again a lifting of $j$.
From a probabilistic perspective, an example of such a functional $M$ is given by the squared covariance:
\begin{align*}
M(F) 
= &\Bigg(
    \int_{M}\!\!\int_{M} 
      x \cdot y \, F(x,y)\, \mathrm{d}\mu(x)\, \mathrm{d}\mu(y) \\
  &\quad
    - \Bigg[\int_{M} x \cdot \proj{\mu}{1}F(x)\, \mathrm{d}\mu(x)\Bigg]
      \Bigg[\int_{M} y \cdot \proj{\mu}{2}F(y)\, \mathrm{d}\mu(y)\Bigg]
  \Bigg)^{2}.
\end{align*}

We now turn to the tensorization of kernels, which will play a key role in the main theorem of this section. 
\begin{definition}\label{TensorKernel}
For $i=1,2$, let $(M_{i},d_{i})$ be a metric space and $k_i$ a kernel on $M_i$, as defined in Subsection \ref{SubSec1.2}. The {\bf tensorization} of $k_{1}$ and $k_{2}$, denoted by $k_{1} \oplus k_{2}$, is given by \begin{align}
&k_{1} \oplus k_{2}((x,y),A):= k_{1}(x,A^{y}) + k_{2}(y,A_{x}),\ (x,y) \in M_{1} \times M_{2}, \notag \\ & \ A \in \calB(M_{1}) \otimes \calB(M_{2}). \notag 
\end{align}
Here, 
\[
A_x=\{y\in M_2:\,(x,y)\in A\}\quad\text{and}\quad
A^y=\{x\in M_1:\,(x,y)\in A\}
\]
denote the $x$- and $y$-sections of $A$, respectively.
\end{definition}
\begin{remark}\label{KernProp}
{\em (i) In the situation of Definition \ref{TensorKernel}, it is in general not clear that $k_{1}\oplus k_{2}$ depends measurably on $(x,y)$. However, if both $k_{1}$ and $k_{2}$ are countable sums of finite kernels, then
$k_{1} \oplus k_{2}$ does possess this property; see the appendix for a proof. We recall that a kernel $k$ on the metric space $M$ is said to be finite if, for every $x \in M$, the measure $k(x,\cdot)$ is finite.  

Moreover, in the situation of Definition \ref{TensorKernel}, it is immediate that for each fixed $(x,y) \in M_{1}\times M_{2}$, the mapping $A\mapsto k_{1} \oplus k_{2}((x,y),A)$ defines a measure on $\calB(M_{1}) \otimes \calB(M_{2})$.

(ii) The notion of kernel tensorization given in Definition \ref{TensorKernel} already appears in the literature. In the discrete case, tensorization of kernels was, for instance, defined and used in \cite[Sec.\ 4]{weber_entropy_2021}. We further refer to \cite{mufa_coupling_1986} and the notes by Jansen \cite{jansen_jump_processes} on jump processes in a very general measure-space setting.
}
\end{remark}

Denoting the operators associated with $k_i$ by $L_i$, $i=1,2$ (cp.\ \eqref{defL}), $k_{1} \oplus k_{2}$ is the appropriate object to represent $L_1\oplus L_2$. This operator
acts on functions $f:\,M_{1} \times M_{2} \rightarrow \bbR$ via 
\begin{equation} \label{OperSum}
((L_1\oplus L_2)f)(x,y) = L_{1}(f^{y})(x) + L_{2}(f_{x})(y),
\end{equation}
where $f_{x}(\cdot ) = f(x,\cdot)$ and $f^{y}(\cdot ) = f(\cdot,y)$, see also \cite[p. 60]{bakry_analysis_2014} for its definition.

\begin{proposition}\label{Sumformulas} 
For $i=1,2$, let $(M_{i},d_{i})$ be a metric space and $k_i$ a kernel on $M_i$.
Let $(x,y)\in M_1\times M_2$ and $f:\, M_{1} \times M_{2} \rightarrow \bbR$ be measurable. Then the following statements hold.
\begin{itemize}
\item [(i)] If $L_{1}(f^y)(x)$ and $L_{2}(f_x)(y)$ exist, then 
\[
((L_{1} \oplus L_{2})f)(x,y)=\int_{M_{1} \times M_{2} \setminus \{(x,y)\}} \!\!\!\big(f(u,v) - f(x,y)\big)\ k_{1} \oplus k_{2}((x,y),\dint{(u,v)}).
\]
\item [(ii)] The $\Psi_\Upsilon$-operator associated with $k_{1} \oplus k_{2}$ is given by 
\begin{align}
\Psi_{\palme}^{k_{1} \oplus k_{2} }(f)(x,y) = \Psi_{\palme}^{k_{1}}(f^{y})(x) + \Psi_{\palme}^{k_{2}}(f_{x})(y).
\end{align}
\end{itemize}
\end{proposition}
Various versions of the first statement can already be found in the literature, see e.g. \cite{jansen_jump_processes} and \cite[Sec.\ 4]{weber_entropy_2021}. For the reader's convenience we provide a sketch of the proof in the appendix, where also an argument for (ii) is given.  

The following lemma is due to Weber and Zacher (\cite[Lemma 2.2]{weber_liyau_2023}) and a key ingredient of our proof of the lifting property of nonlocal Fisher information.
\begin{lemma}\label{main inequlity frederic}
Let $M$ be a metric space and $L$ an operator of the form \eqref{defL}. Let $H: M \times M \rightarrow (0,\infty)$ be such that $H(x,\cdot)$ is $\calB(M)-$measurable and the restriction $H|_{M \setminus \{x\} \times M}$ is $\calB(M \setminus \{x\}) \otimes \calB(M)-$measurable for every $x \in M$. Moreover, let $f:M \rightarrow (0,\infty)$ be $\calB(M)-$measurable. We assume that the integral 
\[
Pf(x) :=  \int_{M} H(x,y) f(y)\ \dint{\nu(y)}
\]
and also $\Psi_{\palme}^{k}(\log(Pf))$ exist for any $x \in M$ and that for $\nu-$a.e. $y \in M$, the expression $\Psi_{\palme}^{k}(\log(H(\cdot,y))(x)$ exists for every $x \in M$. Here $\nu:\, \calB(M) \rightarrow [0,\infty]$ is a $\sigma$-finite measure. Then, for every $x\in M$, we have 
\begin{align}
    \int_{M} \Psi_{\palme}^{k}(\log(H(\cdot,y)))(x)H(x,y)f(y)\ \dint{\nu(y)} \geq \Psi_{\palme}^{k}(\log(Pf))(x) Pf(x). \label{Main inequality}
\end{align}
\end{lemma}
We are now in a position to state and prove our main result concerning the natural lifting of nonlocal Fisher information.
\begin{theorem} \label{ThmLiftingNonlocalFisher}
Let $(M,d)$ be a metric space, $\mu$ a $\sigma$-finite measure on $\calB(M)$, and $k$ a kernel on $M$ which is the countable sum of finite kernels. Then a lifting of the nonlocal Fisher information
\begin{equation} \label{NonlocFish2}
i_{k}(f) = \int_{M} f(x) \Psi_{\palme}^{k}(\log f)(x)\ \dint{\mu(x)}
\end{equation}
is given by 
\[
i_{k \oplus k}(F)=\int_{M\times M}F(x,y)\Psi_\Upsilon^{k\oplus k}(\log F)(x,y)\,\dint{(\mu\otimes\mu)(x,y)}. 
\]
\end{theorem}
\begin{proof}
First, let us show that for any probability density $f:\,M \rightarrow (0,\infty)$ with $i_{k}(f) < \infty$, we have that $i_{k \oplus k}(f \otimes f) = 2 i_{k}(f)$.

By Proposition \ref{Sumformulas} (ii),
\begin{align}
\Psi_{\palme}^{k \oplus k}& \big(\log(f \otimes f)\big)(x,y) = \Psi_{\palme}^{k}(\log(f \otimes f)^{y})(x) + \Psi_{\palme}^{k}(\log(f \otimes f)_{x})(y) \notag \\
& = \int_{\mengdiff{M}{\{x\}}} \palme \big(\log(f(u) f(y)) - \log(f(x) f(y))\big) \ k(x,\dint{u})\notag \\
& \quad + \int_{\mengdiff{M}{\{y\}} }\palme \big(\log(f(x) f(v)) - \log(f(x) f(y))\big) \ k(y,\dint{v})  \notag \\
& = \Psi_{\palme}^{k}(\log f)(x) + \Psi_{\palme}^{k}(\log f)(y). \label{Tensorprop.}
\end{align}
Therefore,
\begin{align}
&i_{k \oplus k}(f \otimes f) = \int_{M\times M}f(x) f(y) \, \Psi_{\palme}^{k \oplus k}(\log(f \otimes f))(x,y) \ \dint{(\mu \otimes \mu)(x,y)} \notag \\
&\overset{\eqref{Tensorprop.}}{=}\int_{M\times M} f(x) f(y)\, \big(\Psi_{\palme}^{k}(\log f)(x) + \Psi_{\palme}^{k}(\log f)(y)\big)   \ \dint{(\mu \otimes \mu)(x,y)} \notag \\
&\overset{\text{Tonelli}}{=} \int_{M} f(y) \int_{M} f(x)\, \Psi_{\palme}^k(\log f)(x) \ \dint{\mu(x)}\, \dint{\mu(y)} \notag \\
&\quad\quad + \int_{M} f(x) \int_{M} f(y)\, \Psi_{\palme}^k(\log f)(y) \ \dint{\mu(y)}\, \dint{\mu(x)} \notag \\
& =   2 \,i_{k}(f).
\end{align}

Next, consider a symmetric function $F:\,M \times M  \rightarrow (0,\infty)$ satisfying $i_{k \oplus k}(F)< \infty$ and $i_{k}(\Pi_\mu F)<\infty$. This in particular implies, that for $\mu$-a.e.\ $x \in M$ and $\mu$-a.e.\ $y \in M$, the terms $\Psi_{\palme}^{k}(\log(\Pi_\mu F))(x)$ and $\Psi_{\palme}^{k}(\log(F(\cdot,y))(x)$ exist. Applying Lemma \ref{main inequlity frederic} with $f = 1$ and $H(x,y) = F(x,y)$ we see that 
$$  \int_{M} \Psi_{\palme}^k(\log(F(\cdot,y)))(x) \, F(x,y) \ \dint{\mu(y)} \geq \Psi_{\palme}^k(\log(\Pi_\mu F))(x) \, \Pi_\mu F(x),$$
for a.a.\ $x \in M$. Thus, a further application of Proposition \ref{Sumformulas} yields
\begin{align*}
&i_{k \oplus k}(F) = \int_{M \times M} F(x,y) \, \Psi_{\palme}^{k \oplus k}(\log F)(x,y) \ \dint{(\mu \otimes \mu)(x,y)} \notag \\
& = \int_{M \times M} F(x,y) \, \big(\Psi_{\palme}^{k }((\log F)^{y})(x) + \Psi_{\palme}^{k}((\log F)_{x})(y)\big)\ \dint{(\mu \otimes \mu)(x,y)} \notag \\
& =\int_{M} \int_{M} F(x,y) \cdot  \Psi_{\palme}^{k}(\log F(\cdot,y))(x) \ \dint{\mu(y)}\, \dint{\mu(x)} \notag \\
&\quad +\int_{M} \int_{M} F(x,y) \,  \Psi_{\palme}^{k}(\log F(x,\cdot))(y) \ \dint{\mu(x)}\, \dint{\mu(y)} \notag \\
&\geq 2 \,  \int_{M} \Psi_{\palme}^{k}(\log(\Pi_\mu F))(x) \, \Pi_\mu F(x)\, \dint{\mu(x)} = 2 \, i_{k}(\Pi_\mu F),
\end{align*} 
where we also make use of the symmetry of $F$. This concludes the proof of the theorem.
\end{proof}
\begin{remark}
{\em
The assumption that the kernel $k$ is a countable sum of finite kernels ensures that
$k\oplus k$ is a kernel on $M\times M$ (again given by a countable sum of finite kernels), see
Remark \ref{KernProp} (i). Moreover, it guarantees the measurability of 
$\Psi_{\palme}^{k}(\log f)$ and 
$\Psi_{\palme}^{k\oplus k}(\log F)$ for measurable functions $f:\,M\to (0,\infty)$ and
$F:\,M\times M\to (0,\infty)$, cf.\ \cite[Lemma 3.2]{kallenberg_foundations_2021}.
}
\end{remark}
\section{Fisher information and $\Gamma$-calculus}\label{Sec3}
\noindent The aim of this section is to study the existence of liftings of the Fisher information in the setting of so-called {\em full Markov triples}, as introduced in \cite{bakry_analysis_2014}. We do not present the full details of this general framework, but instead restrict ourselves to recalling only those elements that are essential for our purposes. 

A full Markov triple consists of a measure space $(E,\calF,\mu)$ with $\sigma$-finite $\mu$, a suitable algebra $\calA$ of real-valued measurable functions on $E$ and an operator $L$ which generates a Markov semigroup
$(P_t)_{t\ge 0}$ of functions on $E$. The {\em carr\'e du champ} operator $\Gamma$ associated with $L$ is defined on $\calA \times \calA$ by 
\[
\Gamma(f,g) = \frac{1}{2}\,\big(L(f g) - g Lf - f  Lg\big), 
\]
and we set $\Gamma(f):=\Gamma(f,f)$, also writing $\Gamma^L$ instead of $\Gamma$ to emphasize the dependence on $L$. Moreover, the algebra $\calA$ is stable under composition with smooth functions $\Phi: \bbR^{k} \rightarrow \bbR$, and $L$ and $\Gamma$ satisfy the {\em diffusion property} 
\begin{align}
& L(\Phi(f_{1},\ldots,f_{k})) =  \sum_{i = 1}^{k} \partial_{i} \Phi(f_{1},\ldots,f_{k}) L(f_{i}) + \sum_{i,j = 1}^{k} \partial_{i}\partial_{j}\Phi(f_{1},\ldots,f_{k}) \Gamma(f_{i},f_{j}),\label{Diffusion property L}
\end{align}
respectively,
\begin{align} 
& \Gamma(\Phi(f_{1},\ldots,f_{k}),g)
 = \sum_{i = 1}^{k} \partial_{i}\Phi(f_{1},\ldots,f_{k}) \Gamma(f_{i},g),\label{Diffusion property Gamma}
 \end{align}
 for all smooth $ \Phi: \bbR^{k} \rightarrow \bbR $ and all $ f_{1},\ldots,f_{k},g \in \calA$. Formally, the diffusion property reflects that $L$ behaves like a second-order differential operator.

The Fisher information associated with $L$ is defined by 
\begin{equation} \label{FishL}
i_{L}(f) = \int_E f\, \Gamma^L\big(\log f\big)\,\dint{\mu}= \int_{E} \frac{\Gamma^{L}(f)}{f} \ \dint{\mu}, 
\end{equation}
see \cite[p.\ 237]{bakry_analysis_2014}. The classical Fisher information \eqref{FishCla} is recovered by choosing $L=\Delta$ on $E=\bbR^d$ with $\mu$ equal to the Lebesgue measure.

We have already seen in Subsection \ref{SubSec1.1} that a natural lifting of the classical Fisher information is given by \begin{align*}
I(F) = \int_{\bbR^{2d}} \frac{|\nabla F(x,y)|^{2}}{F(x,y)} \ \dint{(x,y)} = \int_{\bbR^{2d}} \frac{|\nabla_{x}F(x,y)|^{2}}{F(x,y)} + \frac{|\nabla_{y}F(x,y)|^{2}}{F(x,y)}\ \dint{(x,y)}.
\end{align*}
Using the notation introduced before Proposition \ref{Sumformulas}, we can rewrite
$I(F)$ as
\begin{align*}I(F)  
& = \int_{\bbR^{2d}} \Big( \frac{\Gamma^{\Delta}(F^y)(x)}{F(x,y)} + \frac{\Gamma^{\Delta}(F_x)(y)}{F(x,y)}\Big) \ \dint{(x,y)}
= \int_{\bbR^{2d}} \frac{\Gamma^{\Delta \oplus \Delta}(F)(x,y)}{F(x,y)}\ \dint{(x,y)}\\ & = \int_{\bbR^{2d}} F(x,y)\,{\Gamma^{\Delta \oplus \Delta}(\log F)(x,y)}\ \dint{(x,y)} =i_{\Delta \oplus \Delta}(F),
\end{align*}
where the notation $L\oplus L$ has to be understood as in \eqref{OperSum}.

This suggests that, in the general 
full Markov triple setting, a natural candidate for a lifting of the Fisher information \eqref{FishL}
is given by
\begin{equation} \label{iLL}
i_{L \oplus L}(F) = \int_{E \times E} F\,\Gamma^{L \oplus L}(\log F) \ \dint{(\mu \otimes \mu)}=\int_{E \times E} \frac{\Gamma^{L \oplus L}(F)}{F}\ \dint{(\mu \otimes \mu)}.
\end{equation}
Here, the last equality follows from the fact that, with $\Gamma^L$, also $\Gamma^{L \oplus L}$ satisfies the diffusion property. This is a consequence of the subsequent proposition, see Remark \ref{LLdiff}. 

We show that
$i_{L \oplus L}$ satisfies condition (i) of Definition \ref{LiftingGenDef}. That is, for every admissible function $f$, we have
$i_{L \oplus L}(f \otimes f) = 2 \, i_{L}(f)$. To this end, we first establish the following proposition.
\begin{proposition}\label{TensorpropGamma}
Let $L$ be the operator associated with a full Markov triple. Then, for all suitable functions $F,\,G:\,E \times E\rightarrow \bbR$, we have for $x,y\in E$,
\begin{align} \label{GammaSumProp}
		\Gamma^{L \oplus L}(F,G)(x,y) = \Gamma^{L}(F^y,G^y)(x) + \Gamma^{L}(F_x,G_x)(y).
	\end{align}
\end{proposition}
\begin{proof}
	By definition,
    \begin{align} \Gamma^{L \oplus L}(F,G)(x,y) = \frac{1}{2}\big[& (L\oplus L)(F \, G)(x,y)-((L \oplus L)F)(x,y) G(x,y)  \notag \\ & - F(x,y) ((L\oplus L)G)(x,y)\big]. \label{Gammatensor}\end{align}
	Using the definition of $L \oplus L$ and $\Gamma^L$, we see that 
\begin{align}
		& (L \oplus L)(F \, G)(x,y) = L(F^y\, G^y)(x) + L(F_x\,G_x)(y) \notag \\
		&\quad = G(x,y) L(F^y)(x) + F(x,y) L(G^y)(x) + 2\Gamma^{L}(F^y,G^y)(x) \notag \\
		&\quad\quad + G(x,y) L(F_x)(y) + F(x,y) L(G_x)(y)+2\Gamma^L(F_x,G_x)(y) \label{L+Lid}.
    \end{align}
		Moreover, 
\begin{align}
		 ((L \oplus L)G)(x,y) & = L(G^y)(x) + L(G_x)(y) \label{L+LG}\\
		((L \oplus L)F)(x,y)& = L(F^y)(x) + L(F_x)(y) \label{L+LF}.
\end{align}
Combining \eqref{Gammatensor},\eqref{L+Lid}, \eqref{L+LG}, and \eqref{L+LF} yields the assertion \eqref{GammaSumProp}. 
\end{proof}
\begin{remark} \label{LLdiff}
{\em The carr\'e du champ operator $\Gamma^{L \oplus L}$ also satisfies the diffusion property. For the sake of simplicity, let us restrict to the case $k=1$ in \eqref{Diffusion property Gamma}.
For a smooth function $\Phi:\,\bbR\to \bbR$ and suitable $F,G:\,E\times E\to \bbR$, applying Proposition \ref{TensorpropGamma} twice and using the diffusion property of $\Gamma^L$, we have
\begin{align*}
\Gamma^{L \oplus L}&(\Phi(F),G)(x,y)=\Gamma^L(\Phi(F^y),G^y)(x)+\Gamma^L(\Phi(F_x),G_x)(y) \\
&=\Phi'(F^y)(x)\Gamma^L(F^y,G^y)(x)
+\Phi'(F_x)(y)\Gamma^L(F_x,G_x)(y)\\
&=\Phi'(F(x,y))\big(\Gamma^L(F^y,G^y)(x)+
\Gamma^L(F_x,G_x)(y)\big)\\
&=\Phi'(F(x,y))\,\Gamma^{L \oplus L}(F,G)(x,y).
\end{align*}
}    
\end{remark}

With Proposition \ref{TensorpropGamma} at hand, it is not difficult to show the following theorem.
\begin{theorem}\label{Tensor Fisher Gamma}
Consider a full Markov triple with Markov generator $L$, and let $i_L$ and $i_{L\oplus L}$ be defined by \eqref{FishL} and \eqref{iLL}, respectively. Then, for every admissible probability density $f$ on $E$,
$$  
i_{L\oplus L}(f \otimes f)= 2 \, i_{L}(f).
$$
\end{theorem}
\begin{proof}
Let $f$ be an admissible probability density on $E$. Then, by using Proposition \ref{TensorpropGamma} and the bilinearity of the carr\'e du champ operator, we have
\begin{align}
& \Gamma^{L\oplus L}(f \otimes f)(x,y) =  \Gamma^{L}((f\otimes f)_{x})(y) + \Gamma^{L}((f \otimes f)^{y})(x) \notag \\
&\quad= \Gamma^{L}\big(f(x)\, f(\cdot)\big)(y) + \Gamma^{L}\big(f(\cdot) \,f(y)\big)(x) \notag \\
&\quad=  f(x)^{2} \, \Gamma^{L}(f)(y) + f(y)^{2}
\,  \Gamma^{L}(f)(x). \notag
\end{align}
Combined with an application of Tonelli’s theorem, this yields
\begin{align}
i_{L \oplus L}&(f \otimes f) =   \int_{E} \int_{E} \frac{\Gamma^{L \oplus L}(f \otimes f)(x,y)}{(f \otimes f)(x,y)}\ \dint{\mu(x)} \ \dint{\mu(y)}  \notag \\
&= \int_{E} \int_{E}  \Big(
\frac{f(x)^{2} \, \Gamma^{L}(f)(y)}{f(x)f(y)} + \frac{f(y)^{2}
\,  \Gamma^{L}(f)(x)}{f(x)f(y)}
 \Big)\ \dint{\mu(y)} \ \dint{\mu(x)}\notag \\
&= 2 \, i_{L}(f),
\notag 
\end{align}
where we also use that $\int_E f\,\dint{\mu}=1$.
\end{proof}

Having established that $i_{L \oplus L}$ satisfies condition (i) in the definition of a lifting (of $i_L$), cf.\ Definition \ref{LiftingGenDef}, we now turn to condition (ii). This condition requires that
\begin{equation} \label{openproblem}
i_{L\oplus L}(F)\ge 2 \, i_{L}(\Pi_\mu F)
\end{equation}
holds for all admissible symmetric probability densities $F$ on $E\times E$.  
Without additional assumptions, it is not clear whether \eqref{openproblem} holds in general.

In the case of the classical Fisher information
\eqref{FishCla}, inequality \eqref{openproblem}
follows from the pointwise inequality
\begin{equation} \label{gradungleichung}
\int_{\bbR^d}\frac{\big|\nabla_x F(x,y))\big|^2}{ F(x,y)} \,\dint{y}\ge \frac{\big|\nabla\, \Pi F(x))\big|^2}{ \Pi F(x)},\;\; x\in \bbR^d,
\end{equation}    
where $\Pi F(x)=\int_{\bbR^d}F(x,y)\,\dint{y}$.
Observe that \eqref{gradungleichung} is a special case of \eqref{gradientIE}.
The proof of \eqref{gradungleichung} relies basically only on the interchange of differentiation and integration, together with Hölder’s inequality. For the reader’s convenience, we briefly recall the argument.
\begin{align}
|\nabla_{x}\,\Pi F(x)|^{2} &= \left|\int_{\bbR^d} \nabla_{x} F(x,y) \ \dint{y}\right|^{2}  \le \left(\int_{\bbR^d} \frac{|\nabla_{x} F(x,y)|}{\sqrt{F(x,y)}} \, \sqrt{F(x,y)}\ \dint{y}\right)^{2} \notag \\
&\leq  \int_{\bbR^d} \frac{|\nabla_{x}F(x,y)|^{2}}{F(x,y)} \ \dint{y} \,\,\Pi F(x).\notag 
\end{align}
Dividing by $\Pi F(x)$ yields \eqref{gradungleichung}.

Observe that the proof trivially extends to domains $\Omega\subset \bbR^d$ and general measures $\mu$, that is, we also have
\begin{equation} \label{gradungleichung2}
\int_{\Omega}\frac{\big|\nabla_x F(x,y))\big|^2}{ F(x,y)} \,\dint{\mu(y)}\ge \frac{\big|\nabla\, \Pi_\mu F(x))\big|^2}{ \Pi_\mu F(x)},\;\; x\in \bbR^d,
\end{equation} 
with $\Pi_\mu F(x)=\int_{\Omega}F(x,y)\,\dint{\mu(y)}$.
This allows one to establish the lifting property of the Fisher information $i_L$ for more general second-order operators. We illustrate this with two examples.
The corresponding operators and their associated semigroups can be found in  \cite{bakry_analysis_2014}.
\begin{example} [Laguerre semigroup]
{\em
Let $\alpha >0$. Consider the Laguerre operator $L_{\alpha}$ acting on smooth functions $f $ on $\bbR_{+}$ via $$(L_{\alpha }f)(x) = x f^{\prime \prime}(x) + (x-\alpha)f^{\prime}(x). $$
The carr\'e du champ operator associated with $L_{\alpha}$ is given by $\Gamma(f)(x) = x (f^{\prime}(x))^{2}$.
The corresponding invariant and reversible probability measure is given by $$  \dint{\mu_{\alpha}} = \gamma_{\alpha}^{-1} x^{\alpha -1} e^{-x}  \ \dint{x},$$
where $\gamma_\alpha$ denotes the value of the Gamma function at $\alpha$.
See \cite[Sec.\ 2.7.3]{bakry_analysis_2014}.

Let us now consider the associated Fisher information $$ i_{L_{\alpha}}(f) = \int_{\bbR_{+}} x \, \frac{(f^{\prime}(x))^{2}}{f(x)}\ \dint{\mu_{\alpha}(x)}.$$
As discussed above, a natural candidate for a lifting of $i_{L_{\alpha}}$ is given by $i_{L_{\alpha} \oplus L_{\alpha}}$. In view of Theorem \ref{Tensor Fisher Gamma}, it remains to establish \eqref{openproblem} in order to conclude that $i_{L_{\alpha} \oplus L_{\alpha}}$ is indeed a lifting of $i_{L_{\alpha}}$.
For an admissible symmetric probability density $F:\,\bbR_{+}^{2} \rightarrow \bbR$ with respect to $\mu_\alpha\otimes \mu_\alpha$, we have, using Proposition \ref{TensorpropGamma}, \begin{align*}
 &i_{L_{\alpha} \oplus L_{\alpha}}(F) =   \int_{\bbR_{+}} \int_{\bbR_{+}} \frac{\Gamma^{L_{\alpha}}(F^y)(x) + \Gamma^{L_{\alpha}}(F_x)(y)}{F(x,y)}\ \dint{\mu_{\alpha}(x)}\ \dint{\mu_{\alpha}(y)}\notag \\
&\quad = \int_{\bbR_{+}} \int_{\bbR_{+}} \Big(\frac{x\,  (\partial_{x}F(x,y))^{2}}{F(x,y)} + \frac{y\, (\partial_{y}F(x,y))^{2}}{F(x,y)}\Big)\ \dint{\mu_{\alpha}(x)} \ \dint{\mu_{\alpha}(y)} \\
&\quad= \int_{\bbR_{+}} x \int_{\bbR_{+}} \frac{(\partial_{x}F(x,y))^{2}}{F(x,y)} \  \dint{\mu_{\alpha}(y)}\  \dint{\mu_{\alpha}(x)}\\
&\quad\quad\quad+  \int_{\bbR_{+}} y \int_{\bbR_{+}} \frac{(\partial_{y}F(x,y))^{2}}{F(x,y)} \  \dint{\mu_{\alpha}(x)}\  
\dint{\mu_{\alpha}(y)} \\
& \quad\overset{\eqref{gradungleichung2}}{\geq}   \int_{\bbR_{+}} x \, \frac{(\partial_{x}\Pi_{\mu_{\alpha}}F(x))^{2}}{\Pi_{\mu_{\alpha}}F(x)} \, \dint{\mu_{\alpha}(x)} + \int_{\bbR_{+}} y \, \frac{(\partial_{y}\Pi_{\mu_{\alpha}}F(y))^{2}}{\Pi_{\mu_{\alpha}}F(y)}\,  \dint{\mu_{\alpha}(y)}\\
&\quad =2\,i_{L_{\alpha}}(\Pi_{\mu_\alpha}F).
\end{align*}
Hence, the desired property \eqref{openproblem} is indeed satisfied.
}
\end{example}
\begin{example}[Jacobi semigroup]
{\em Let $\alpha,\beta>0$.
	Consider the Jacobi operator acting on smooth functions $f$ on $[-1,1]$ given by $$L_{\alpha,\beta}(f)(x) = (1-x^{2})f^{\prime \prime}(x)  - [(\alpha  + \beta )x + \alpha - \beta]f^{\prime}(x). $$ The corresponding invariant and reversible probability measure is given by $$ \dint{\mu_{\alpha,\beta}(x)} = C_{\alpha,\beta} (1-x)^{\alpha-1}(1+x)^{\beta-1}\ \dint{x},$$
    with a suitable constant $C_{\alpha,\beta}$.
	The carr\'e du champ operator takes the form $$  \Gamma(f)(x) = (1-x^{2})(f^{\prime}(x))^{2}.$$
    See \cite[Sec.\ 2.7.4]{bakry_analysis_2014}.
	Similar calculations as in the previous example show that $i_{L_{\alpha,\beta} \oplus L_{\alpha,\beta}}$ is a lifting of $i_{L_{\alpha,\beta}}.$
    }
\end{example}
\section{Nonlocal-to-local limit of the fractional Fisher information}\label{Sec4}
\noindent In this and the following section we consider the nonlocal Fisher information $i_s\,(:=i_{k_s})$ associated
with the (negative) fractional Laplace operator $L=-(- \Delta)^s$ on $\bbR^d$ with $s\in (0,1)$. The corresponding kernel $k_s$ is given by \eqref{fractionalkernel}, and $i_s(f)$ has the form
\begin{align*}
 i_{s}(f) & = c(d,s) \int_{\bbR^{d}} \int_{\bbR^{d}} \frac{f(x) \palme(\log(f(x+h)) - \log(f(x)))}{|h|^{d+2s}} \ \dint{h}\, \dint{x}.
\end{align*}

It is well known that, for suitable $f$,
$$ \lim_{s\to 1} (-\Delta )^{s}f(x) = -\Delta f(x),\ x \in \bbR^{n},$$ see e.g. \cite[Theorem 12.4]{stinga_regularity_2024}. The following theorem provides a corresponding result for the fractional Fisher information $i_s$.
\begin{theorem} \label{NonlocalToLocal}
Let $f \in L^{1}(\bbR^{d})\cap C^{2}(\bbR^{d})$ satisfy the following properties: 
\begin{itemize} 
    \item[(i)] There exist constants $\beta \in (0,2)$, $c_{0} \in (0,1)$ and $\delta > 0$ such that \[
    f(x) \geq c_{0}e^{-\delta|x|^{\beta}},\quad x \in \bbR^{d}.
    \]
    \item[(ii)] There exist $R_{0}, c_{1},c_{2}> 0 $ such that 
    \[
    c_{1}f(x) \leq f(x+h) \leq c_{2} f(x),\quad \text{if}\;\,|x|\geq R_{0},|h|\leq 1.
    \]
    \item[(iii)] $\displaystyle \int_{\bbR^{d}} f(x) (|x|^{\beta} + | \log(f(x))|) \ \dint{x}< \infty.$
    \item[(iv)] $\nabla \sqrt{f} \in L^{2}(\bbR^{d}).$ 
\end{itemize}
Then for every $s \in \left( \frac{\beta}{2},1 \right)$, we have $i_{s}(f) < \infty$ and
\[
\lim_{s\to 1} i_{s}(f) = i(f).
\]
\end{theorem}
\begin{proof}
We begin by showing that $i_{s}(f)< \infty$ for all $s \in \left( \frac{\beta}{2},1\right)$. 
We write $i_{s}(f)$ as 
\begin{align*}
 i_{s}(f)
 & = c(d,s) \bigg( \int_{\bbR^{d}} \int_{B_{1}(0)} \frac{f(x) \palme(\log(f(x+h)) - \log(f(x)))}{|h|^{d+2s}} \ \dint{h} \dint{x}  \\
& \qquad + \int_{\bbR^{d}} \int_{B_{1}(0)^{c}} \frac{f(x) \palme(\log(f(x+h)) - \log(f(x)))}{|h|^{d+2s}} \ \dint{h} \dint{x}\bigg)\\
& =: i_{s,1}(f) + i_{s,2}(f). 
\end{align*} 

We first estimate $i_{s,2}(f)$. Note that
\begin{align}\label{abPalme}
0\le a\Upsilon(\log b-\log a)=b-a-a\big(\log b-\log a\big),\quad a,b>0.
\end{align}
In particular, we have
\begin{equation*}
a\Upsilon(\log (b)-\log (a))\le b+a\log (a)+a\big(-\log (b)\big),\quad a,b>0,
\end{equation*}
which will be used below for $b<1$, and
\begin{equation*}
a\Upsilon(\log (b)-\log (a))\le b+a\log a,\quad a>0,\,b\ge 1.
\end{equation*}
Moreover, by assumption (i), 
\[
-\log (f(x))\leq -\log (c_{0}) + \delta|x|^{\beta},\quad x\in \bbR^{d}.
\]
For $x \in \bbR^{d}$ we set $A_{x}(f) = \{h \in \bbR^{d}\ :\ f(x+h)\geq 1\}.$ We estimate 
\begin{align*}
&\frac{ i_{s,2}(f)}{c(d,s)}
\leq  \int_{\bbR^d}\int_{B_1(0)^C\cap A_x(f)}\,\frac{f(x+h)+f(x)\log(f(x))}{|h|^{d+2s}}\, \dint{h} \dint{x}\\
&\quad\quad + \int_{\bbR^d}\int_{B_1(0)^C\cap A_x(f)^C}
\!\!\!\!\!\frac{f(x+h)+f(x)\log(f(x))-f(x)\log(f(x+h))
}{|h|^{d+2s}}\, \dint{h} \dint{x}\\
&\quad =  \int_{B_1(0)^C} \frac{1}{|h|^{d+2s}} \int_{\bbR^d} f(x+h) \dint{x} \dint{h}+ \int_{\bbR^d}f(x)\log( f(x)) \dint{x}\int_{B_1(0)^C}\frac{\dint{h}}{|h|^{d+2s}}\\
& \quad\quad+\int_{\bbR^d}\int_{B_1(0)^C\cap A_x(f)^C}\frac{f(x)\big(-\log(f(x+h))\big)
}{|h|^{d+2s}}\, \dint{h} \dint{x}\\
&\quad\leq \omega_{d-1}\big(\norm{f}{1}+\norm{f\log(f)}{1}\big)\int_{1}^{\infty} \frac{\dint{\rho}}{\rho^{1+2s}}\\
& \quad\quad +\int_{\bbR^d} \int_{B_1(0)^C}\frac{f(x)\big(-\log(c_0)+\delta|x+h|^{\beta}\big)}
{|h|^{d+2s}}\,\dint{h} \dint{x},
\end{align*}
where $\omega_{d-1}$ denotes the surface area of the unit sphere in $\bbR^d$. Note that 
\[ 
|x+h|^\beta\leq 2^\beta\big(|x|^\beta+|h|^\beta\big)\le 4\big(|x|^\beta+|h|^\beta\big),
\]
and thus
\begin{align*}
\int_{\bbR^d} \int_{B_1(0)^C}&\frac{f(x)\big(-\log( c_0)+\delta|x+h|^\beta\big)}
{|h|^{d+2s}}\,\dint{h}\ \dint{x}\\
 &\leq \omega_{d-1} \Big(\log\big( c_0^{-1}\big) \norm{f}{1}+4\delta \int_{\bbR^d}f(x)|x|^\beta\,dx\Big) \int_1^\infty \frac{\dint{\rho}}{\rho^{1+2s}}\\
& \quad+4\delta\omega_{d-1}\norm{f}{1} 
\int_1^\infty \frac{\dint{\rho}}{\rho^{1+2s-\beta}}.
\end{align*}
Overall, we obtain an estimate of the form
\begin{align}
i_{s,2}(f) & \leq c(d,s) \tilde{C}\Big(\norm{f}{1}+\norm{f\log(f)}{1}+\int_{\bbR^d}f(x)|x|^\beta\,dx\Big)\Big(1+\int_{1}^\infty \!\!\!\!\frac{\dint \rho}{\rho^{1+2s-\beta}}\Big), \label{upperest i2}
\end{align}
where $\tilde{C}=\tilde{C}(d,\delta_0,c_0,\beta)$ does not depend on $s\in(\frac{\beta}{2},1)$. This shows $i_{s,2}(f)<\infty$.

Next, we want to estimate $i_{s,1}(f)$. Let $R \geq R_{0}$. We write
\begin{align*}
i_{s,1}(f) & =  c(d,s) \int_{B_{R}(0)} \int_{B_{1}(0)} \frac{f(x) \palme(\log(f(x+h))-\log(f(x))}{|h|^{d+2s}}\ \dint{h}\, \dint{x} \\
&\quad+ c(d,s) \int_{B_{R}(0)^{c}} \int_{B_{1}(0)} \frac{f(x) \palme(\log(f(x+h))-\log(f(x))}{|h|^{d+2s}}\ \dint{h}\, \dint{x}\\
&=:  i_{s,1,a}(f) + i_{s,1,b}(f).
\end{align*}
By the continuity of $f$ and assumption (ii), there exist constants $0<M_{1}<M_{2},$ depending only on $f,c_1,c_2$ and $R_{0}$, such that $\frac{f(x+h)}{f(x)} \in [M_{1},M_{2}]$ for all $x \in \bbR^{d}$ and $h \in B_{1}(0).$ Thus we find a constant $c_{4}=c_4(M_1,M_2)>0$ such that for all $x\in B_R(0)$ and $h\in B_1(0)$,
\[ 
\palme(\log(f(x+h))-\log(f(x)) \leq c_{4} \big(\log(f(x+h))-\log(f(x))\big)^{2}. 
\]
Using this, together with the estimate
\begin{align*}
\big(\log(f(x+h))-\log(f(x))\big)^{2} 
\leq  \int_{0}^{1} |\nabla \log(f(x+th))|^{2}\ \dint{t} \cdot |h|^{2},
\end{align*}
we obtain
\begin{align*}
&\frac{i_{s,1,a}(f)}{c(d,s)}  \leq  c_{4}  \int_{B_{1}(0)} \frac{1}{|h|^{d-2+2s}} \int_{0}^{1} \int_{B_{R}(0)} f(x) \cdot |\nabla \log(f(x+th))|^{2}\ \dint{x}\, \dint{t}\, \dint{h}\\
&\quad\leq  \,\frac{c_{4}}{M_{1}} \int_{B_{1}(0)} \frac{1}{|h|^{d-2+2s}} \int_{0}^{1} \int_{B_{R}(0)}f(x+th) \cdot |\nabla \log(f(x+th))|^{2}\ \dint{x}\, \dint{t}\, \dint{h}\\
&\quad\leq  \, \frac{4 c_{4}\omega_{d-1}}{M_{1}} \norm{\nabla \sqrt{f}}{2}^{2}\int_{0}^{1} \frac{1}{\rho^{2s-1}}\ \dint{\rho}
= \frac{2 c_{4}\omega_{d-1}}{M_{1}(1-s)} \norm{\nabla \sqrt{f}}{2}^{2}< \infty, 
\end{align*}
thanks to (iv).

Since, by (ii), $\frac{f(x+ h)}{f(x)} \in [c_{1},c_{2}]$ for all $x\in B_{R}(0)^{c}$ and $h \in B_{1}(0)$, the same argument as above shows that there exists a constant $c_{5}> 0$ independent of $s$ and $R$ such that 
\begin{align}
i_{s,1,b}(f) \leq \frac{2 c(d,s)c_{5}\omega_{d-1}}{c_{1}(1-s)} \norm{\nabla \sqrt{f}}{L^2(B_{R-1}(0)^{c})}^{2} < \infty. \label{upperest i_s,b}
\end{align}
Altogether, this shows $i_{s}(f)< \infty$ for all $s \in (\frac{\beta}{2},1)$.

Next, we prove the convergence $i_{s}(f)\to i(f)$ as $s \rightarrow 1$. Let $\varepsilon> 0$ be arbitrary and fix $s_{0} \in \left( \frac{\beta}{2},1 \right)$. In view of \eqref{upperest i2}, we have
\begin{align*}
0\le i_{s,2}(f) 
\leq  c(d,s) M(f),\quad s\in [s_0,1),
\end{align*}
for a suitable constant $M(f)$ independent of $s\in [s_0,1)$. Since $\lim_{s \to 1} \frac{c(d,s)}{1-s} = \frac{4d}{\omega_{d-1}}$, see \cite[Corollary 4.2]{di_nezza_hitchhikers_2012}, it follows that
\[
i_{s,2}(f)\to 0\quad \text{as}\;\,s\to 1.
\]

Since $\nabla \sqrt{f} \in L^{2}(\bbR^{d})$, there is $R\geq R_{0} $ such that 
$\norm{\nabla \sqrt{f}}{L^2(B_{R-1}(0)^{c})}^{2}
 < \varepsilon$. Thus, due to \eqref{upperest i_s,b}, we have 
 \begin{align*}
     i_{s,1,b}(f) \leq \frac{2 c(d,s)c_{5}\omega_{d-1}}{c_{1}(1-s)} \,\varepsilon \leq M  \varepsilon, 
\end{align*}
for a suitable constant $M =M(c_{1},c_{5},d)$. Here we use that $\frac{c(d,s)}{1-s}$ is bounded on $[s_{0},1)$.

We now turn to $i_{s,1,a}(f)$. Since $f \in C^{2}$ we have 
\[
\log(f(x+h)) - \log(f(x)) = \scalprod{\nabla \log(f(x))}{h}+ R(x,h),
\]
with the remainder term satisfying 
\[ 
|R(x,h)| \leq \tilde{M} |h|^{2},\quad x \in B_{R}(0),\,h\in B_1(0),
\]
where $\tilde{M}$ only depends on $c_0,\delta,\beta$ and the $C^2(\overline{B_{R+1}(0)})$-norm of $f$.

Moreover, by Taylor's theorem, we have 
$\palme(r) = \frac{1}{2} r^{2} + O(r^{3})$
for small $r$, and thus
\[ 
\palme\big(\log(f(x+h)) - \log(f(x))\big) = \frac{1}{2}\scalprod{\nabla \log(f(x))}{h}^{2} + \tilde{R}(x,h), 
\] 
where 
\[ 
|\tilde{R}(x,h)| \leq  \hat{M} |h|^{3},\quad x \in B_{R}(0),\,h\in B_1(0),
\] 
for a suitable constant $\hat{M}$.
Therefore,
\begin{align*}
{i_{s,1,a}(f)} & = \frac{c(d,s)}{2}\int_{B_{R}(0)} \int_{B_{1}(0)} \frac{f(x)\scalprod{\nabla \log(f(x))}{h}^{2}}{|h|^{d+2s}}\ \dint{h} \,\dint{x} \\
&+  c(d,s) \int_{B_{R}(0)} \int_{B_{1}(0)} \frac{f(x)\tilde{R}(x,h)}{|h|^{d+2s}}\ \dint{h}\,\dint{x}=:j_{s,1}(f)+j_{s,2}(f).
\end{align*}
Observe that 
\begin{align*}
 |j_{s,2}(f)|& \leq c(d,s) 
 \hat{M}|B_{R}(0)| \norm{f}{L^\infty(B_R(0))}  \int_{B_{1}(0)} \frac{|h|^{3}}{|h|^{d+2s}}\ \dint{h} \\
& = \frac{c(d,s) 
 \hat{M}\omega_{d-1}}{3-2s}|B_{R}(0)| \norm{f}{L^\infty(B_R(0))} \rightarrow 0 \quad\text{as}\;\,s\to 1,
\end{align*}
because $c(d,s)\to 0$ as $s\to 1$.

Turning to $j_{s,1}(f)$, we have 
\begin{align*}
\int_{B_{1}(0)} \!\!\!\frac{\scalprod{\nabla \log(f(x))}{h}^{2}}{|h|^{d+2s}}\ \dint{h} & = \sum_{i,j = 1}^{d} \partial_{x_{i}} \log(f(x)) \partial_{x_{j}}\log(f(x)) \int_{B_{1}(0)} \frac{h_{i}h_{j}}{|h|^{d+2s}}\ \dint{h}\\
& =  \sum_{i = 1}^{d} (\partial_{x_{i}} \log(f(x)))^{2} \int_{B_{1}(0)} \frac{h_{i}^{2}}{|h|^{d+2s}} \ \dint{h} \\
& = \frac{1}{d} \sum_{i = 1}^{d} (\partial_{x_{i}} \log(f(x)))^{2} \int_{B_{1}(0)} \frac{\dint{h}}{|h|^{d+2s-2}}\   \\
&= \frac{\omega_{d-1}}{2d(1-s)} |\nabla \log(f(x))|^{2}.
\end{align*}
Hence
\begin{align*}
j_{s,1}(f) =\frac{c(d,s)}{4d(1-s)}\omega_{d-1}  \int_{B_{R}(0)} f(x)|\nabla \log(f(x))|^{2}\ \dint{x}. 
\end{align*}

Finally, writing $i(f)=\int_{B_R(0)}\ldots+\int_{B_R(0)^c}\ldots$, we may estimate as
\begin{align}
|i_{s}(f) - i(f)| & =|i_{s,2}(f)+i_{s,1,a}(f)+i_{s,1,b}(f)-i(f)|\nonumber\\
& \le i_{s,2}(f)+i_{s,1,b}(f)+|j_{s,2}(f)|+4\norm{\nabla \sqrt{f}}{L^2(B_{R}(0)^{c})}^{2}\nonumber\\
&\quad +4\Big| \frac{c(d,s)}{4d(1-s)}\omega_{d-1}-1\Big|  \norm{\nabla \sqrt{f}}{2}^{2}.   \label{FinalEstimate}  
\end{align} 
Recall that $0\le i_{s,2}(f)+i_{s,1,b}(f)+|j_{s,2}(f)|\to 0$ and $ \frac{c(d,s)}{1-s} \to \frac{4d}{\omega_{d-1}}$ as $s\to 1$ as well as
$\norm{\nabla \sqrt{f}}{L^2(B_{R-1}(0)^{c})}^{2}
 < \varepsilon$.
Hence, for $s$ sufficiently close to $1$, the right-hand side of \eqref{FinalEstimate} can be estimated by $C\varepsilon$ with some suitable constant $C$ independent of $\varepsilon$. This proves $\lim_{s \to 1} i_{s}(f)  = i(f)$.
\end{proof}

\section{Nonlocal Blachman-Stam inequality}\label{Sec5}
\noindent An important property of the classical Fisher information is the {\em Blachman-Stam inequality}. It states that if $f,g$ are suitable probability densities on $\bbR^{d}$, then for every $\alpha \in (0,1)$,
$$i(f_{\sqrt{\alpha}} \ast g_{\sqrt{1-\alpha}}) \leq \alpha \,  i(f) + (1-\alpha) \, i(g), $$
where $\ast$ denotes convolution and where for a constant $c > 0$, the rescaled density $f_c$ is given by $f_{c}(x) = \frac{1}{c^{d}} f\left(\frac{x}{c}\right)$, $x \in \bbR^{d}$. See, for example, \cite[Theorem 7]{carlen_superadditivity_1991}.

A crucial role in the proof of the Blachman-Stam inequality given in \cite{carlen_superadditivity_1991} is played by the inequality \begin{align} \int_{\bbR^{d}}\left|\nabla_{x}\left( \int_{\bbR^{d}} g(x,y) \ \dint{y}\right)^{\frac{1}{2}} \right|^{2} \ \dint{x} \leq \int_{\bbR^{d}} \int_{\bbR^{d}} \left|\nabla_{x}g(x,y)^{\frac{1}{2}}\right|^{2} \ \dint{y}\, \dint{x},  \label{carlen} \end{align}
 valid for suitable  positive functions $g$, compare \cite[Theorem 2]{carlen_superadditivity_1991}.
 Using that 
 \[
 4|\nabla \sqrt{g}|^{2} = \, \frac{|\nabla g|^{2}}{g} =  |\nabla \log(g)|^{2} g,
 \]
 inequality \eqref{carlen} can be rewritten as \begin{align}
     &\int_{\bbR^{d}} \left|\nabla_{x}\log \left( \int_{\bbR^{d}} g(x,y) \ \dint{y}\right)\right|^{2} \, \left(\int_{\bbR^{d}} g(x,y) \ \dint{y} \right)\ \dint{x}  \notag \\ & \leq \int_{\bbR^{d}}\int_{\bbR^{d}} |\nabla_{x} \log g(x,y)|^{2} \,g(x,y)\ \dint{y}\ \dint{x}.\label{carlen2}
 \end{align}
Note that \eqref{carlen2} is an integrated version of the pointwise estimate \eqref{gradientIE} with $f=1$ and $H(x,y)=g(x,y)$. Since a corresponding nonlocal version of \eqref{carlen2} is provided by the key Lemma \ref{main inequlity frederic} (see also \eqref{PsiUpsIE} in Section 1), it is natural to expect that an analogous Blachman-Stam inequality holds for the Fisher information associated with the fractional Laplacian. This is precisely the content of the following theorem.
\begin{theorem} \label{ThmBSI}
Let $s \in (0,1)$ and $f,g$ be two probability densities on $\bbR^d$ satisfying $i_{s}(f),i_{s}(g)< \infty$. Then
\begin{align}
    i_{s}(f_{\sqrt{\alpha}} \ast g_{\sqrt{1-\alpha}}) \leq \alpha^{s}
    \, i_{s}(f) + (1-\alpha)^{s} \,  i_{s}(g),\quad \alpha\in (0,1). \label{Improved Blachman Stam}
\end{align}
\end{theorem}
\begin{proof} Our proof is inspired by the line of arguments used in the proof of \cite[Theorem 7]{carlen_superadditivity_1991}. For ease of notation, we will write $\Psi_{\palme}^{s}$ instead of $\Psi_{\palme}^{k_{s}}$ in the following.

Let $f,g$ be two probability densities and fix $\alpha \in (0,1).$ For $x,y \in \bbR^{d}$, we set
\begin{align*}
 L_{\alpha}(x,y)& = \sqrt{\alpha}\,x- \sqrt{1-\alpha}\,y,\\
 B_{\alpha}(x,y) & = \sqrt{1-\alpha}\,x + \sqrt{\alpha}\,y,\\
 \rho(x,y) &= f\big(L_{\alpha}(x,y)\big) \, g\big(B_{\alpha}(x,y)\big).
\end{align*}

First, as in the proof given by Carlen in \cite{carlen_superadditivity_1991}, we rewrite $(f_{\sqrt{\alpha}} \ast g_{\sqrt{1-\alpha}})(x)$ for $x \in \bbR^{d}$ by means of the substitution $ z = \frac{y}{\sqrt{\alpha} \cdot \sqrt{1-\alpha}} - \frac{\sqrt{1-\alpha}}{\sqrt{\alpha}}x$.
\begin{align*}
 &(f_{\sqrt{\alpha}} \ast g_{\sqrt{1-\alpha}})(x) = \int_{\bbR^{d}} \frac{1}{\sqrt{\alpha}^{d}}\, f\left(\frac{x-y}{\sqrt{\alpha}} \right) \, \frac{1}{\sqrt{1-\alpha}^{d}}\, g\left(  \frac{y}{\sqrt{1-\alpha}}\right)\  \dint{y} \\
&= \int_{\bbR^{d}} f(\sqrt{\alpha}\,x - \sqrt{1-\alpha}\,z) \, g(\sqrt{1-\alpha}\,x + \sqrt{\alpha}\,z)\ \dint{z}=\int_{\bbR^{d}} \rho(x,y)\,\dint{y}.
\end{align*}
Applying Lemma \ref{main inequlity frederic} with $H(x,y) = \rho(x,y)$ yields
\begin{align*} 
& \Psi_{\palme}^{s}\left( \log\left( f_{\sqrt{\alpha}} \ast g_{\sqrt{1-\alpha}} \right)\right)(x) \cdot (f_{\sqrt{\alpha}} \ast g_{\sqrt{1-\alpha}})(x)  \notag \\
&=  \Psi_{\palme}^{s}\left(   \log\left(\int_{\bbR^{d}}\rho(\cdot,y)  \ \dint{y} \right)\right)(x)  \, \int_{\bbR^{d}}\rho(x,y)\ \dint{y} \notag \\
& \leq \int_{\bbR^{d}}  \Psi_{\palme}^{s}\left(\log(\left( \rho(\cdot,y)\right) \right)(x) \,   \rho(x,y) \ \dint{y},
\end{align*}
which in turn, by integrating over $x\in \bbR^d$, implies
\begin{align}
& i_{s}(f_{\sqrt{\alpha}} \ast g_{\sqrt{1-\alpha}}) = \int_{\bbR^{d}}  \Psi_{\palme}^{s}\left( \log\left( f_{\sqrt{\alpha}} \ast g_{\sqrt{1-\alpha}} \right)\right)(x) \cdot (f_{\sqrt{\alpha}} \ast g_{\sqrt{1-\alpha}})(x)\ \dint{x} \notag \\
&\leq \int_{\bbR^{d}} \int_{\bbR^{d}}  \Psi_{\palme}^{s}\left(\log \rho(\cdot,y) \right)(x) \cdot   \rho(x,y) \ \dint{y}\,\dint{x}.\label{Finale Ungleichung}
\end{align}

Next, we rewrite the integrand on the right-hand side of inequality \eqref{Finale Ungleichung} in an appropriate way. Due to identity \eqref{abPalme},
\begin{align}
&\rho(x,y) \, \palme\big( \log \rho(\tilde{x},y) - \log \rho(x,y)\big)  \notag \\ 
& = \rho(\tilde{x},y) - \rho(x,y) - \rho(x,y) \big( \log \rho(\tilde{x},y) -\log \rho(x,y) \big) \notag \\
& = f(L_{\alpha}(\tilde{x},y)) g(B_{\alpha}(\tilde{x},y)) -  f(L_{\alpha}(x,y)) g(B_{\alpha}(x,y)) \notag \\
& \quad - f(L_{\alpha}(x,y)) g(B_{\alpha}(x,y))\Big[ \log(f(L_{\alpha}(\tilde{x},y))) - \log(f(L_{\alpha}(x,y)))  \notag \\ & \qquad \qquad \qquad \qquad \qquad \quad + \log(g(B_{\alpha}(\tilde{x},y))) - \log(g(B_{\alpha}(x,y))) \Big]. \label{intermediate result}
\end{align}
Adding and subtracting $f(L_{\alpha}(\tilde{x},y))g(B_{\alpha}(x,y))$ and  $f(L_{\alpha}(x,y)) g(B_{\alpha}(\tilde{x},y))$ on the right-hand side of the relation \eqref{intermediate result} and rearranging the terms we find
\begin{align*}
&\rho(x,y)\,\palme\bigl(\log\rho(\tilde{x},y)-\log\rho(x,y)\bigr) \notag \\
&= g(B_\alpha(x,y))  \Bigl(
      f(L_\alpha(\tilde{x},y)) - f(L_\alpha(x,y)) \notag\\
&\quad - f(L_\alpha(x,y)) \bigl(\log f(L_\alpha(\tilde{x},y)) - \log f(L_\alpha(x,y))\bigr)
      \Bigr) \notag\\
&\quad + f(L_\alpha(x,y))  \Bigl(
      g(B_\alpha(\tilde{x},y)) - g(B_\alpha(x,y)) \notag\\
&\quad - g(B_\alpha(x,y)) \bigl(\log g(B_\alpha(\tilde{x},y)) - \log g(B_\alpha(x,y))\bigr)
      \Bigr) \notag\\
&\quad + f(L_\alpha(\tilde{x},y)) g(B_\alpha(\tilde{x},y)) 
      + f(L_\alpha(x,y)) g(B_\alpha(x,y)) \notag\\
&\quad - f(L_\alpha(x,y)) g(B_\alpha(\tilde{x},y)) 
      - f(L_\alpha(\tilde{x},y)) g(B_\alpha(x,y)) \notag \\
&=g(B_\alpha(x,y)) f(L_{\alpha}(x,y)) \palme\bigl(
      \log f(L_\alpha(\tilde{x},y)) - \log f(L_\alpha(x,y)) \bigl) \notag\\
&\quad + f(L_\alpha(x,y)) g(B_{\alpha}(x,y)) \palme\bigl(
      \log g(B_\alpha(\tilde{x},y)) - \log g(B_\alpha(x,y))\bigr) \notag\\
&\quad  + \bigl(f(L_{\alpha}(\tilde{x},y))-f(L_{\alpha}(x,y))\bigr) \cdot \bigl(g(B_{\alpha}(\tilde{x},y))-g(B_{\alpha}(x,y))\bigr),
\end{align*}
where in the last step we used again identity \eqref{abPalme}. 
For fixed $\varepsilon> 0$ it follows that
\begin{align}
& \int_{\bbR^{d}} \int_{\bbR^{d}} \int_{\bbR^{d}} \indicator{\bbR^{d} \setminus B_{\varepsilon}(x)}(\tilde{x}) \rho(x,y)\, \frac{\palme\left( \log\rho(\tilde{x},y) -\log\rho(x,y)\right)}{|x-\tilde{x}|^{d+2s}}\ \dint{\tilde{x}}\, \dint{y}\, \dint{x} \notag \\
&= \int_{\bbR^{d}} \int_{\bbR^{d}} \int_{\bbR^{d}} \indicator{\bbR^{d} \setminus B_{\varepsilon}(x)}(\tilde{x}) \Big[g(B_{\alpha}(x,y)) f(L_{\alpha}(x,y)) \frac{\palme\left(\log\left(\frac{f(L_{\alpha}(\tilde{x},y))}{f(L_{\alpha}((x,y))}\right)\right)}{|x-\tilde{x}|^{d+2s}} \notag \\
& \qquad \qquad \qquad \qquad \qquad   +  f(L_{\alpha}(x,y)) g(B_{\alpha}(x,y)) \frac{\palme\left(\log\left(\frac{g(B_{\alpha}(\tilde{x},y))}{g(B_{\alpha}(x,y))}\right)\right)}{|x-\tilde{x}|^{d+2s}}   \notag \\
&   + \frac{\left( f(L_{\alpha}(\tilde{x},y)) - f(L_{\alpha}(x,y))\right)\, \left(  g(B_{\alpha}(\tilde{x},y))-  g(B_{\alpha}(x,y))\right)}{|x-\tilde{x}|^{d+2s}}\Big]\ \dint{\tilde{x}}\,\dint{y} \,\dint{x}. \label{identity triple integral}
\end{align}
Observe that, by positivity of $f$ and $g$,
\begin{align*}
&\Theta(x,y,\tilde{x}):=\big|f(L_{\alpha}(\tilde{x},y)) - f(L_{\alpha}(x,y))\big|\,\big|g(B_{\alpha}(\tilde{x},y))-  g(B_{\alpha}(x,y))\big|\\
& \le f(L_{\alpha}(\tilde{x},y))\, g(B_{\alpha}(\tilde{x},y)) +f(L_{\alpha}(\tilde{x},y))\, g(B_{\alpha}(x,y))\\
& \quad+f(L_{\alpha}(x,y))\, g(B_{\alpha}(h+x,y)) + f(L_{\alpha}(x,y))\, g(B_{\alpha}(x,y)). 
\end{align*}
Using this and the substitution $h=\tilde{x}-x$, we obtain
\begin{align}
&\int_{\bbR^{d}} \int_{\bbR^{d}} \int_{\bbR^{d}} \indicator{\bbR^{d} \setminus B_{\varepsilon}(x)}(\tilde{x})\,\frac{\Theta(x,y,\tilde{x})}{|x-\tilde{x}|^{d+2s}}\,\dint{\tilde{x}}\,\dint{y}\, \dint{x}\nonumber\\
&  \le   \int_{\bbR^{d}} \int_{\bbR^{d}} \int_{\bbR^{d}} \indicator{\bbR^{d} \setminus B_{\varepsilon}(0)}(h) \Big[\frac{f(L_{\alpha}(h + x,y)) g(B_{\alpha}(h+x,y)) }{|h|^{d+2s}}  \notag \\
& \qquad \qquad \qquad+ \frac{f(L_{\alpha}(h+x,y)) g(B_{\alpha}(x,y))+
f(L_{\alpha}(x,y)) g(B_{\alpha}(h+x,y))}{|h|^{d+2s}}\nonumber\\
&\qquad\qquad\qquad+ \frac{f(L_{\alpha}(x,y)) g(B_{\alpha}(x,y))}
{|h|^{d+2s}}\Big]\ \dint{h}\, \dint{y}\, \dint{x}.
 \label{integral ident 1}
\end{align}
We then change variables by setting $u = L_\alpha(x,y)$ and $v = B_\alpha(x,y)$.
Note that the corresponding Jacobian determinant equals $1$. Therefore, the integral on the right of \eqref{integral ident 1} is equal to
\begin{align}
 & \int_{\bbR^{d}} \int_{\bbR^{d}} \int_{\bbR^{d}} \indicator{\bbR^{d} \setminus B_{\varepsilon}(0)}(h) \Big[\frac{f(\sqrt{\alpha}\,h +u)\, g(\sqrt{1-\alpha}\,h + v) +f(\sqrt{\alpha}\,h + u)\, g(v)}{|h|^{d+2s}}  \notag \\
& \qquad \qquad \qquad+ \frac{f(u)\, g(\sqrt{1-\alpha}\,h + v) + f(u) g(v) }{|h|^{d+2s}} \Big] \ \dint{h}\, \dint{u}\, \dint{v} \notag \\
&= 4 \, \int_{\bbR^{d}} \indicator{\bbR^{d} \setminus B_{\varepsilon}(0)}(h) \, \frac{1}{|h|^{d+2s}}\ \dint{h}< \infty, \label{Fubini allowed}
\end{align}
where it was used that $f$ and $g$ are probability densities. 

Consequently, employing again the substitutions $h = \tilde{x} - x$,
$u = L_\alpha(x,y)$ and $v = B_\alpha(x,y)$ and changing the order of integration, which is justified by \eqref{Fubini allowed}, we find
\begin{align*}
& \int_{\bbR^{d}} \int_{\bbR^{d}} \int_{\bbR^{d}}   \indicator{\bbR^{d} \setminus B_{\varepsilon}(x)}(\tilde{x})\frac{\left( f(L_{\alpha}(\tilde{x},y)) - f(L_{\alpha}(x,y))\right)}{|x-\tilde{x}|^{d+2s}}\,\times\nonumber\\
& \qquad\qquad\qquad\qquad\qquad\times \big(  g(B_{\alpha}(\tilde{x},y))-  g(B_{\alpha}(x,y))\big)\ \dint{\tilde{x}}\,\dint{y}\, \dint{x} \notag \\
&= \int_{\bbR^{d}} \int_{\bbR^{d}} \int_{\bbR^{d}}   \indicator{\bbR^{d} \setminus B_{\varepsilon}(0)}(h) \frac{\left( f(\sqrt{\alpha}h + u) - f(u)\right) \left(  g(\sqrt{1-\alpha} + v)-  g(v)\right)}{|h|^{d+2s}} \ \dint{h}\,\dint{u} \,\dint{v} \notag \\
&=  \int_{\bbR^{d}} \frac{\indicator{\bbR^{d} \setminus B_{\varepsilon}(0)}(h)}{|h|^{d+2s}} \int_{\bbR^{d}}\left(  g(\sqrt{1-\alpha}h + v)-  g(v)\right)\times\nonumber\\
&\qquad\qquad\qquad\qquad\times \int_{\bbR^{d}}  \left( f(\sqrt{\alpha}h + u) - f(u)\right)\ \dint{u}\, \dint{v}\, \dint{h} = 0, \label{mixed term zero}
\end{align*}
because $f$ is a probability density. 

Thus, for every $\varepsilon > 0$, we have
\begin{align}
&\int_{\bbR^{d}} \int_{\bbR^{d}} \int_{\bbR^{d}} \indicator{\bbR^{d} \setminus B_{\varepsilon}(x)}(\tilde{x}) \rho(x,y) \, \frac{\palme\big( \log(\rho(\tilde{x},y)) -\log(\rho(x,y))\big)}{|x-\tilde{x}|^{d+2s}}\ \dint{\tilde{x}}\, \dint{y}\, \dint{x}  \notag \\
&=  \int_{\bbR^{d}} \int_{\bbR^{d}} \int_{\bbR^{d}} \indicator{\bbR^{d} \setminus B_{\varepsilon}(x)}(\tilde{x}) \Big[g(B_{\alpha}(x,y)) f(L_{\alpha}(x,y))\, \frac{\palme\left(\log\left(\frac{f(L_{\alpha}(\tilde{x},y))}{f(L_{\alpha}((x,y))}\right)\right)}{|x-\tilde{x}|^{d+2s}} \notag \\
& \qquad \qquad \qquad \qquad \qquad   +  f(L_{\alpha}(x,y)) g(B_{\alpha}(x,y))\, \frac{\palme\left(\log\left(\frac{g(B_{\alpha}(\tilde{x},y))}{g(B_{\alpha}(x,y))}\right)\right)}{|x-\tilde{x}|^{d+2s}} \Big]\ \dint{\tilde{x}}\,  \dint{y}\, \dint{x} \notag\\
&= \int_{\bbR^{d}} \int_{\bbR^{d}} \int_{\bbR^{d}} \indicator{\bbR^{d} \setminus B_{\varepsilon}(0)}(h) \Big[g(v) f(u) \frac{\palme\left(\log\left(\frac{f(\sqrt{\alpha}\,h + u)}{f(u)}\right)\right)}{|h|^{d+2s}} \notag \\
& \qquad \qquad \qquad \qquad \qquad   +  f(u) g(v) \frac{\palme\left(\log\left(\frac{g(\sqrt{1-\alpha}\,h + v))}{g(v)}\right)\right)}{|h|^{d+2s}} \Big] \ \dint{h}\, \dint{u}\, \dint{v}.\nonumber
\end{align}
Multiplying this identity by $c(d,s)$ and sending $\varepsilon \rightarrow 0$, we obtain
\begin{align}
& \int_{\bbR^{d}} \int_{\bbR^{d}} \Psi_{\palme}^{s}(\log(\rho(\cdot,y))(x) \, \rho(x,y) \ \dint{y}\,\dint{x} \notag \\
&= c(d,s)\int_{\bbR^{d}} \int_{\bbR^{d}} \int_{\bbR^{d}}  \Big[g(v) f(u) \frac{\palme\left(\log\left(\frac{f(\sqrt{\alpha}h + u)}{f(u)}\right)\right)}{|h|^{d+2s}} \notag \\
& \qquad \qquad \qquad \qquad    +  f(u) g(v) \frac{\palme\left(\log\left(\frac{g(\sqrt{1-\alpha}h + v))}{g(v)}\right)\right)}{|h|^{d+2s}} \Big] \ \dint{h}\, \dint{u}\, \dint{v} \notag \\
&= c(d,s)  \int_{\bbR^{d}} f(u)\int_{\bbR^{d}}  \frac{\palme\left(\log\left(\frac{f(\sqrt{\alpha}\,h + u)}{f(u)}\right)\right)}{|h|^{d+2s}}\ \dint{h}\, \dint{u} \notag \\
& \quad\quad  + c(d,s)  \int_{\bbR^{d}} g(v)\int_{\bbR^{d}} \frac{\palme\left(\log\left(\frac{g(\sqrt{1-\alpha}\,h + v)}{g(v)}\right)\right)}{|h|^{d+2s}}\ \dint{h}\, \dint{v} \notag \\
& \underset{w = \sqrt{1-\alpha}\,h}{\overset{z = \sqrt{\alpha}\,h}{=}} c(d,s)  \int_{\bbR^{d}} f(u)\int_{\bbR^{d}}  \frac{\palme\left(\log\left(\frac{f(z + u)}{f(u)}\right)\right)}{|z|^{d+2s}} \sqrt{\alpha}^{2s}\ \dint{z}\, \dint{u} \notag \\
& \quad\quad  + c(d,s)  \int_{\bbR^{d}} g(v)\int_{\bbR^{d}} \frac{\palme\left(\log\left( \frac{g(w + v)}{g(v)}\right)\right)}{|w|^{d+2s}} \sqrt{1-\alpha}^{2s}\ \dint{w}\, \dint{v} \notag \\
&= \alpha^{s}  \int_{\bbR^{d}} f(u) \Psi_{\palme}^{s}(\log f)(u)\ \dint{u } + (1-\alpha)^{s} \int_{\bbR^{d}} g(v) \Psi_{\palme}^{s}(\log g)(v)\ \dint{v}  \notag \\
&= \alpha^{s} \, i_{s}(f )  + (1-\alpha)^{s} \, i_{s}(g). \notag 
\end{align}
This finishes the proof of the theorem.
\end{proof}
\section{Appendix}
\subsection{Lifting property of entropy}\label{A1}
Let $(E,\calA,\mu)$ be a $\sigma$-finite measure space. We show that a lifting of the entropy 
\begin{equation*}
    h(f)=\int_E f \log f\,\dint{\mu},
\end{equation*}
defined for suitable probability densities $f$ on $E$ with respect to $\mu$, is given by the entropy on the product space, 
\[
H(F) = \int_{E \times E}\!\! F \ \log F\ \dint{\mu \otimes \mu}, 
\]
defined for suitable probability densities $F$ on $E\times E$ with respect to $\mu\oplus \mu$.
\begin{proof}
We first show condition (i) in the definition of a lifting. For every admissible probability density $f$ on $E$, we have 
\begin{align*} 
H(f \otimes f) & = \int_{E} \int_{E} f(x) f(y)\, \log\big(f(x) f(y)\big) \ \dint{\mu(y)} \ \dint{\mu(x)}\\ 
& = \int_{E}f(x) \int_{E} f(y)\, \log(f(y))\ \dint{\mu(y)} \ \dint{\mu(x)}  \notag \\ 
&\quad  + \int_{E} f(y)\int_{E} f(x) \, \log(f(x)) \ \dint{\mu(x)}\ \dint{ \mu(y)} = 2 \, h(f). \end{align*}

In order to show that the pair $(h,H)$ fulfills property (ii) of the definition of a lifting, we follow the ideas of the proofs of \cite[Thm.\ 2.2.1 and Thm.\ 2.6.5]{cover_elements_2006}. Let $F$ be an admissible symmetric probability density on $E\times E$ with respect to $\mu \otimes \mu$. Then 
\begin{align*}
& H(F) = \int_{E} \int_{ E} F(x,y) \, \log\left(\frac{F(x,y)}{\Pi_{\mu}F(x)} \, \Pi_\mu F(x)\right) \ \dint{\mu(x)}\, \dint{\mu(y)} \\
&= \int_{E} \int_{E} F(x,y) \ \dint{\mu(y)}\, \log\big(\Pi_\mu F(x)\big)\, \dint{\mu(x)}  \\
&\quad  + \int_{E} \int_{E} F(x,y) \, \log\left(\frac{F(x,y)}{\Pi_\mu F(x)}\right) \ \dint{\mu(x)}\, \dint{\mu(y)}=:H_1(F)+H_2(F),
\end{align*}
and
\[
H_1(F)=\int_{E} \Pi_\mu F(x) \log\big(\Pi_\mu F(x)\big)\, \dint{\mu(x)}=h(\Pi_\mu F).
\]
Moreover,
\begin{align*}
 H_2(F) & =  \int_{E} \int_{E} F(x,y) \, \log\left(\frac{F(x,y)}{\Pi_\mu F(x) \, \Pi_\mu F(y)}\, \Pi_\mu F(y)\right) \ \dint{\mu(x)}\, \dint{\mu(y)}\\
&=  \int_{E} \int_{E} F(x,y)\, \dint{\mu(x)}\, \log\big(\Pi_\mu F(y)\big) \ \dint{\mu(y)} \\ 
&\quad + \int_{E} \int_{E} F(x,y) \, \log\left(\frac{F(x,y)}{\Pi_\mu F(x) \,\Pi_\mu F(y)}\right) \ \dint{\mu(x)}\, \dint{\mu(y)}\\
& =:H_{2a}(F)+H_{2b}(F),
\end{align*}
and, analogously to $H_1(F)$, we have 
$H_{2a}(F)=h(\Pi_\mu F)$. Concerning $H_{2b}(F)$,
we exploit the convexity of $-\log$ and apply Jensen's inequality, thereby obtaining
\begin{align*}
H_{2b}(F) &= \int_{E} \int_{E} F(x,y) \, \left[-\log\left(\frac{\Pi_{\mu}F(x) \, \Pi_{\mu}F(y)}{F(x,y)}\right)\right] \ \dint{\mu(x)}\, \dint{\mu(y)} \\
\geq & -\log\left(\int_{E} \int_{E} F(x,y) \, \frac{\Pi_{\mu}F(x) \, \Pi_{\mu}F(y)}{F(x,y)} \ \dint{\mu(x)} \dint{\mu(y)} \right) = \log(1) = 0.
\end{align*}
Overall, this shows that $H(F) \geq 2 \, h(\Pi_\mu F).$
\end{proof}
\subsection{Proof of Remark \ref{KernProp} (i)}
In order to show the asserted measurability property of the kernel $k_{1}\oplus k_{2}$, it is sufficient to prove the statement for finite kernels.
\begin{proposition}
    \label{tensorized kernel appendix}
For $i = 1,2$, let $(M_{i},d_{i})$ be a metric space with kernel $k_{i}$. Moreover, assume that for every $x \in M_{i}$, the measure $k_{i}(x,\cdot)$ is finite, $i = 1,2$. Then $k_{1}\oplus k_{2}$ is a kernel on $M_{1}\times M_{2}$.
\end{proposition}
\begin{proof}
It is clear that for fixed $(x,y) \in M_{1} \times M_{2}$, $k_{1} \oplus k_{2}((x,y),\cdot )$ is a $\sigma$-finite measure on $ \calB(M_{1}) \otimes \calB(M_{2})$. 
    
It remains to show that $k_{1} \oplus k_{2}$ depends measurably on $(x,y)$. To this end, define 
\begin{align*}&\mathcal{H} := \bigg\{ A \in \calB(M_{1}) \otimes \calB(M_{2}):(k_{1} \oplus k_{2})(\cdot , A) \text{ is measurable} \bigg\}.\end{align*}
We claim that $\mathcal{H}$ is a $\lambda$-system which contains the $\pi$-system $\calB(M_{1}) \times \calB(M_{2})$. To see the latter, let $A = B \times C \in \calB(M_{1}) \times \calB(M_{2}).$  Observe that for $x\in M_1$ and $y \in M_{2}$,
      \begin{align*}
		(B\times C)^{y} = \begin{cases}
			B &, \text{ if }y \in C,\\
			\emptyset &,\text{ else},
		\end{cases}
		\quad\text{and}\quad
		&(B\times C)_{x} = \begin{cases}
			C &,\text{ if }x \in B,\\
			\emptyset &,\text{ else}.
		\end{cases}
	\end{align*}
	Thus, by Definition \ref{TensorKernel},
    \begin{align*}
		 k_{1} \oplus k_{2}((x,y),A) & = k_{1}(x,(B\times C)^{y}) + k_{2}(y,(B\times C)_{x})\\
        &= k_{1}(x,B) \cdot \indicator{C}(y) + k_{2}(y,C) \cdot \indicator{B}(x),
	\end{align*}
	which is $\calB(M_{1}) \otimes \calB(M_{2})$-measurable, since $k_{1}(\cdot,B)$ and $k_{2}(\cdot ,C)$ depend measurably on $x$ and $y$, respectively. Thus $A \in \mathcal{H}$, and so $\calB(M_{1}) \times \calB(M_{2}) \subset \mathcal{H}$.
    
     We next show that $\mathcal{H}$ is a $\lambda$-system. Clearly, $M_{1} \times M_{2} \in \mathcal{H}$. $\mathcal{H}$ is also closed under nested differences, i.e. if $A,B \in \mathcal{H},\ A \subset B$, also $B \setminus A \in \mathcal{H}$. Indeed, letting $A,B \in \mathcal{H}$ with $A \subset B$, we have, for every $x \in M_{1}$ and $y \in M_{2}$, 
    $$ (B \setminus A)^{y} = B^{y}\setminus A^{y} \text{ and } (B \setminus A)_{x} = B_{x}\setminus A_{x}.$$
    Thus, recalling that $k_{1}(x,\cdot)$ and $k_{2}(y,\cdot)$ are finite for $x \in M_{1},\ y \in M_{2}$, 
    \begin{align*}
        k_{1} \oplus k_{2}&((x,y),B \setminus A) = k_{1}(x,B^{y} \setminus A^{y}) + k_{2}(y,B_{x} \setminus A_{x}) \\
        &= k_{1}(x,B^{y}) - k_{1}(x,A^{y}) + k_{2}(y,B_{x}) - k_{2}(y,A_{x})
    \end{align*}
    is measurable. Hence $B \setminus A \in \mathcal{H}$. Moreover, $\mathcal{H}$ is closed under increasing limits. In fact, 
let $(A_{n})_{n \in \bbN} \subset \mathcal{H}$ with $A_{n} \nearrow A.$ Then for fixed $x \in M_{1}$ and $y \in M_{2}$, we have $(A_{n})^{y} \nearrow A^{y}$ and $(A_{n})_{x} \nearrow A_{x}$. Since every measure is continuous from below,
\begin{align*}
    & k_{1} \oplus k_{2}((x,y),A) = \limes{n}{ \infty} k_{1} \oplus k_{2}((x,y),A_{n})     = \limes{n}{\infty} \big[k_{1}(x,A_{n}^{y}) + k_{2}(y,(A_{n})_{x}) \big],
    \end{align*} 
    which is again measurable, and so $A \in \mathcal{H}$. This shows that $\mathcal{H}$ is a $\lambda$-system.

    By the monotone class theorem (or $\pi$-$\lambda$ theorem), see e.g. \cite[Theorem 1.1]{kallenberg_foundations_2021}, it follows that $\calB(M_{1}) \otimes \calB(M_{2})  = \sigma(\calB(M_{1}) \times \calB(M_{2})) \subset \mathcal{H}$. 
    Hence, for every $A \in \calB(M_{1})\otimes \calB(M_{2})$, $k_{1}\oplus k_{2}(\cdot,A)$ depends measurably on $(x,y)$.
\end{proof}
\begin{remark}
 {\em
 The kernel $k_s$ defined by \eqref{fractionalkernel} with $s\in (0,1)$ is a countable sum of finite kernels. Indeed, it can be written as 
\[
k_{s}(x,dy) = \sum_{n = 1}^{\infty} \indicator{\{ \frac{1}{n+1}\leq |x-y|< \frac{1}{n}\}}(y)\frac{c(d,s)}{|x-y|^{d+2s}}\ \dint{y} + \indicator{\{1 \leq |x-y|\}}(y)\frac{c(d,s)}{|x-y|^{d+2s}}\ \dint{y}. 
\]
Hence, by means of Proposition \ref{tensorized kernel appendix}, we can conclude that $k_{s} \oplus k_{s}$ is again a kernel.
}
\end{remark}
\subsection{Proof of Proposition \ref{Sumformulas}}
We only provide a sketch of the proof. For $i=1,2$, let $(M_{i},d_{i})$ be a metric space and $k_i$ a kernel on $M_i$. We set $M=M_1\times M_2$. Let $(x,y)\in M$ and $g:\, M \rightarrow \bbR$ be measurable.

In order to prove Proposition \ref{Sumformulas} (i), we first consider a nonnegative function $g$ and show, by means of a measure-theoretic induction argument, that 
\begin{align}
& \int_{M\setminus \{(x,y)\}} g(u,v) \ k_{1} \oplus k_{2}((x,y), \dint{(u,v)}) \notag \\& =  \int_{M_{1}\setminus \{x\}} g(u,y) \ k_{1}(x,\dint{u}) + \int_{M_{2} \setminus \{y\}} g(x,v) \ k_{2}(y,\dint{v}). \label{factorization}
\end{align}
Indeed, for $g=\indicator{E}$ with $E \in \call{B}(M_{1}) \otimes \calB(M_{2})$, and abbreviating $z:=(x,y)$, we have
\begin{align*}
 &\int_{M\setminus \{z\}} g(u,v) \ k_{1} \oplus 
 k_{2}(z,\dint{(u,v)}) 
 =\int_{M}  \indicator{E \cap (M \setminus \{z\})}(u,v)  \ k_{1} \oplus 
 k_{2}(z,\dint{(u,v)})\\
 &= k_{1} \oplus k_{2}(z,E\cap (M \setminus \{z\}))= k_{1}(x,(E\cap (M \setminus \{z\}))^{y})+k_{2}(y,(E\cap (M \setminus \{z\}))_{x})\\
 &=\int_{M_{1} \setminus \{x\}}  \indicator{E^{y}}(u) \ k_{1}(x,\dint{u}) + \int_{M_{2} \setminus \{y\}} \indicator{E_{x}}(v) \ k_{1}(x,\dint{u}),
\end{align*}
which equals the right-hand side of \eqref{factorization}. Here we use that
 $(E\cap (M \setminus \{(z)\}))^{y} = E^{y} \cap (M_{1} \setminus \{x\})$ and $\indicator{E^{y}}(u) = \indicator{ E}(u,y)$, and the analogous relations for the $x$-sections. It then follows immediately that \eqref{factorization} extends to simple functions. Using this, 
\eqref{factorization} with measurable $g\ge 0$ is obtained by approximation with simple functions.

Let now $f:\,M\to \bbR$ be measurable and assume that $L_{1}(f^y)(x)$ and $L_{2}(f_x)(y)$ exist.
Then assertion (i) in Proposition \ref{Sumformulas} (in the case of nonsingular kernels) is obtained by applying \eqref{factorization} to the positive and negative part of the function $g:\, M \rightarrow \bbR$, $g(u,v)=f(u,v) - f(x,y)$.
The singular case can be treated by similar, albeit more technical, arguments.

Assertion (ii) in Proposition \ref{Sumformulas} 
follows directly from \eqref{factorization} when applied to the nonnegative, measurable function 
\[
g(u,v)=\palme\big(f(u,v)-f(x,y)\big),\quad (u,v)\in M.
\]

\printbibliography

\end{document}